\documentclass[12pt,numbers]{elsarticle}
\journal{}

\makeatletter
\def\ps@pprintTitle{%
 \let\@oddhead\@empty
 \let\@evenhead\@empty
 \def\@oddfoot{\hfill\thepage}%
 \let\@evenfoot\@oddfoot}
\makeatother

\usepackage{amsmath,amssymb,amsfonts,amsthm,graphicx}
\usepackage[bookmarksnumbered,colorlinks=true]{hyperref}
\usepackage[labelfont=bf]{caption}

\providecommand{\doi}[1]{\href{https://doi.org/#1}{DOI:#1}}
\usepackage{xurl} 
\renewcommand{\doi}[1]{%
 \href{https://doi.org/#1}{\nolinkurl{DOI:#1}}%
}

\usepackage{dsfont} 
\usepackage{enumitem} 
\usepackage{mathtools} 
\usepackage{appendix} 
\usepackage{geometry} 
\usepackage{subdepth}

\numberwithin{equation}{section}

\theoremstyle{plain}
\newtheorem{theorem}{Theorem}[section]

\newtheorem{lemma}[theorem]{Lemma}

\newcommand{\N}{\mathbb{N}}
\newcommand{\R}{\mathbb{R}}
\newcommand{\PP}{\mathsf{P}}
\newcommand{\EE}{\mathsf{E}}

\newcommand{\bb}[1]{\boldsymbol{#1}}
\newcommand{\ii}{\mathrm{i}}
\newcommand{\rd}{\mathrm{d}}

\newcommand{\leqdef}{\vcentcolon=}
\newcommand{\reqdef}{=\vcentcolon}
\newcommand{\ind}{\mathds{1}}

\makeatletter\newcommand{\setword}[2]{\phantomsection #1\def\@currentlabel{#1}\label{#2}}\makeatother

\allowdisplaybreaks

\begin{document}

\begin{frontmatter}

\title{\textbf{Strict Total Positivity of a Gauss Hypergeometric Kernel, and the Sharp Threshold for the One-Dimensional \\ \texorpdfstring{$\boldsymbol{SL(2,\R)}$}{SL2R} Conformal Block}}

\author[a1]{Fr\'ed\'eric Ouimet}\ead{frederic.ouimet2@uqtr.ca}
\author[a2]{Donald Richards}\ead{richards@stat.psu.edu}

\address[a1]{Universit\'e du Qu\'ebec \`a Trois-Rivi\`eres, Trois-Rivi\`eres, QC G8Z 4M3, Canada}
\address[a2]{Department of Statistics, Penn State University, University Park, PA 16802, USA\vspace{-5mm}}

\begin{abstract}
In this paper, we completely resolve the two total positivity problems raised by Li [J. High Energy Phys., 2023(7):Paper No.~167, 44 pp., 2023], thereby determining the precise scope of this positivity structure in the one-dimensional conformal bootstrap. First, the Gauss hypergeometric kernel $\mathcal{F}(\Delta,z) = {}_2F_1(\Delta,\Delta;2\Delta;z)$ is proved to be strictly totally positive of order infinity for $\Delta > 0$ and $z \in (0,1)$. Second, the sharp lower $\Delta$-parameter threshold for the associated one-dimensional $SL(2,\R)$ conformal block $G_{\Delta}(z) = z^{\Delta}\mathcal{F}(\Delta,z)$ is shown to be $1/2$: the conformal-block kernel $G_{\Delta}(z)$ is strictly totally positive of order infinity for $\Delta \geq 1/2$. For every $\tau \in (0,1/2)$, there exists a strictly negative odd-order minor of $G_{\Delta}(z)$ such that all its $\Delta$-values are in $(\tau,1/2)$ and all its $z$-values can be chosen arbitrarily close to one. Consequently, no restriction $z > z_0$ with $z_0 < 1$ can restore total positivity of order infinity over all $\Delta > 0$.
\end{abstract}


\end{frontmatter}

\section{Introduction}\label{sec:intro}

Many physical systems display the same patterns when observed at very different length scales. This behavior appears near continuous phase transitions, such as those in fluids and magnets, and in quantum field theory, the mathematical framework used to describe fundamental particles. Conformal field theory provides a common mathematical framework for these systems. The conformal bootstrap tests which of these models can exist by imposing basic consistency requirements, including that a quantity calculated in two different but equivalent ways must have the same value; see, e.g., \citet{PolandRychkovVichi2019}. For example, the conformal bootstrap has produced precise estimates for the three-dimensional Ising model, a basic model of how a magnet behaves near the temperature at which it loses its magnetism; see \citet{KosEtAl2016}.

While seeking a mathematical explanation for strong numerical bounds in one-dimensional conformal field theories, \citet{Li2023} was led to study the total positivity of the hypergeometric and conformal-block kernels denoted throughout this paper by
\[
\mathcal{F}(\Delta,z) \leqdef {}_2F_1(\Delta,\Delta;2\Delta;z), \qquad G_{\Delta}(z) \leqdef z^{\Delta}\mathcal{F}(\Delta,z), \qquad \Delta > 0, \quad 0 < z < 1,
\]
where ${}_2F_1(a,b;c;z)$ denotes the Gauss hypergeometric function defined in \eqref{eq:gaussian.hypergeometric.function} below.

Let $\N \leqdef \{1,2,\ldots\}$. Let $I,J \subseteq \R$ be intervals and let $K:I \times J \to \R$. For $n \in \N$, the kernel $K$ is \emph{totally positive of order $n$} (denoted $\mathrm{TP}_n$) if
\begin{equation}\label{eq:tp.n}
\det\!\big[K(x_i,y_j)\big]_{i,j=1}^m \geq 0
\end{equation}
for every $m \in \{1,\ldots,n\}$, $x_1 < \cdots < x_m$ in $I$, and $y_1 < \cdots < y_m$ in $J$. The kernel $K$ is strictly totally positive of order $n$ ($\mathrm{STP}_n$) if all the determinants in \eqref{eq:tp.n} are strictly positive. The kernel $K$ is $\mathrm{TP}_{\infty}$ (respectively, $\mathrm{STP}_{\infty}$) if it is $\mathrm{TP}_n$ (respectively, $\mathrm{STP}_n$) for every $n \in \N$.

Based on numerical and asymptotic evidence, \citet[Section~3.2, pp.~14--15]{Li2023} conjectured, in the terminology used here, that $\mathcal{F}(\Delta,z)$ is $\mathrm{STP}_{\infty}$ on $(0,\infty) \times (0,1)$. The question for $G_{\Delta}(z)$ is more delicate: although $z^{\Delta}$ is strictly positive, its dependence on both variables suggests that it could change the signs of the determinants. Indeed, \citet[Eqs.~(3.28)--(3.29), p.~19]{Li2023} reported a strictly negative $3 \times 3$ determinant with magnitude about $10^{-5654}$. Such a small violation of positivity is invisible at ordinary numerical precision and shows why an exact analysis is needed. Li's work therefore left open two questions: whether $\mathcal{F}(\Delta,z)$ is $\mathrm{STP}_{\infty}$ and what the sharp lower $\Delta$-parameter threshold is for the total positivity of $G_{\Delta}(z)$.

Both questions are settled in the present article. First, the Gauss hypergeometric kernel $\mathcal{F}(\Delta,z)$ is proved to be $\mathrm{STP}_{\infty}$ on $(0,\infty) \times (0,1)$. Second, the sharp lower $\Delta$-parameter threshold for the full conformal block $G_{\Delta}(z)$ is shown to be $1/2$: for every $\tau > 0$, the kernel $G_{\Delta}(z)$ is $\mathrm{STP}_{\infty}$ on $[\tau,\infty) \times (0,1)$ if and only if $\tau \geq 1/2$. The sharpness statement is stronger in that, whenever $\tau < 1/2$, there exist strictly negative odd-order minors with all their $\Delta$-values in $(\tau,1/2)$ and all their $z$-values arbitrarily close to one. Thus, there is no corresponding lower threshold $z_0 < 1$ that can restore $G_{\Delta}(z)$ to a kernel that is $\mathrm{TP}_{\infty}$ over the full range $\Delta > 0$.

In proving that the hypergeometric factor $\mathcal{F}(\Delta,z)$ is $\mathrm{STP}_{\infty}$, we transform Euler's integral representation for the Gauss hypergeometric function into an expectation involving a family of random variables whose probability density functions can be written as compositions of elementary kernels that are $\mathrm{STP}_{\infty}$.

For the full $SL(2,\R)$ conformal block $G_{\Delta}(z)$, an exact representation in terms of a Legendre function of the second kind leads, by means of an induction based on Wronskians, to the $\mathrm{STP}_{\infty}$ property on the region $\{(\Delta,z) \in [1/2,\infty) \times (0,1)\}$. Sharpness is deduced from the zero-balanced expansion of Gauss' hypergeometric function near $z = 1$ together with a moment argument that establishes the existence of a strictly negative odd-order minor below the threshold.

The rest of the paper is organized as follows. Section~\ref{sec:definitions} introduces the remaining special-function notation used throughout the paper. Section~\ref{sec:general.tools} collects the scaling, change-of-variable, and composition principles for total positivity needed in the proofs. Section~\ref{sec:TP.F} proves the strict total positivity of the Gauss hypergeometric kernel $\mathcal{F}(\Delta,z)$ through a probabilistic integral representation. Section~\ref{sec:TP.G} establishes the sharp threshold for the conformal-block kernel $G_{\Delta}(z)$, treating sufficiency and sharpness separately. Appendix~\ref{app:Wronskian.tools} supplies the Legendre-function and Wronskian arguments used to establish sufficiency for the threshold $\tau \geq 1/2$, while Appendix~\ref{app:boundary.calculations} contains the expansion, moment, and determinant calculations used to prove its sharpness for $\tau < 1/2$.

\section{Definitions and notation}\label{sec:definitions}

Throughout the paper, we use the standard notations for the classical special functions and related quantities; cf.~\citet{AskeyRoy2010} or \citet{Daalhuis2010}. Thus, for $a,b > 0$, the classical gamma and beta functions are defined, respectively, by
\[
\Gamma(a) \leqdef \int_0^{\infty} x^{a-1}e^{-x} \, \rd x, \qquad B(a,b) \leqdef \frac{\Gamma(a)\Gamma(b)}{\Gamma(a + b)}, \qquad a,b > 0.
\]

For $a > 0$, a random variable $X$ is said to have a \textit{gamma distribution with index parameter} $a$, denoted $X \sim \operatorname{Gamma}(a)$, if the probability density function of $X$ is
\[
f_a(x) = \frac{1}{\Gamma(a)} x^{a-1} e^{-x}, \qquad x > 0.
\]
Similarly, a random variable $X$ is said to have a \textit{beta distribution with index parameter} $a$, denoted $X \sim \operatorname{Beta}(a,1)$, if the probability density function of $X$ is
\[
b_a(x) = \frac{1}{B(a,1)} x^{a-1} = a x^{a-1}, \qquad 0 < x < 1.
\]

Euler's constant is
\[
\gamma \leqdef \lim_{n \to \infty} \bigg[\sum_{j=1}^n j^{-1} - \log n\bigg] = 0.57721\ldots,
\]
and the digamma function is denoted by
\[
\psi(a) \leqdef \frac{\rd}{\rd a} \log\Gamma(a) = \frac{\Gamma'(a)}{\Gamma(a)}, \qquad a > 0.
\]

Let $\N_0 \leqdef \N \cup \{0\}$. For $a \in \R$, the rising factorial is defined as
\[
(a)_0 \leqdef 1, \qquad (a)_k \leqdef a(a + 1)\cdots(a + k - 1), \qquad k \in \N.
\]
For $a,b \in \R$, $c \in \R\setminus\{0,-1,-2,\ldots\}$, and $|z| < 1$, the Gauss hypergeometric function is
\begin{equation}
\label{eq:gaussian.hypergeometric.function}
{}_2F_1(a,b;c;z) \leqdef \sum_{k=0}^{\infty}\frac{(a)_k(b)_k}{(c)_k}\frac{z^k}{k!}.
\end{equation}

\section{General tools for total positivity}\label{sec:general.tools}

The following two invariance principles and the continuous Cauchy--Binet formula will be used repeatedly throughout the paper.

\begin{lemma}[Scaling by strictly positive functions]\label{lem:scaling.functions}
Let $n \in \N \cup \{\infty\}$, let $I,J \subseteq \R$ be intervals, let $K:I \times J \to \R$ be a kernel, and let $f:I \to (0,\infty)$ and $g:J \to (0,\infty)$. Define
\[
\widetilde{K}(x,y) \leqdef f(x) g(y) K(x,y), \qquad x \in I, \quad y \in J.
\]
Then, for every $m \in \N$, $x_1 < \cdots < x_m$ in $I$, and $y_1 < \cdots < y_m$ in $J$,
\begin{equation}\label{eq:positive.scaling.minor}
\det\!\big[\widetilde{K}(x_i,y_j)\big]_{i,j=1}^m = \left(\prod_{i=1}^m f(x_i)\right)\left(\prod_{j=1}^m g(y_j)\right)\det\!\big[K(x_i,y_j)\big]_{i,j=1}^m.
\end{equation}
The two determinants in \eqref{eq:positive.scaling.minor} therefore have the same sign. Consequently, $K$ is $\mathrm{TP}_n$ (respectively, $\mathrm{STP}_n$) if and only if $\widetilde{K}$ is $\mathrm{TP}_n$ (respectively, $\mathrm{STP}_n$).
\end{lemma}

\begin{proof}
This is \citet[Chap.~1, \S2, Thm.~2.1(a), pp.~17--18]{Karlin1968} with strictly positive factors.
\end{proof}

\begin{lemma}[Strictly increasing changes of variables]\label{lem:increasing.transformation}
Let $n \in \N \cup \{\infty\}$, let $I,J,\widetilde{I},\widetilde{J} \subseteq \R$ be intervals, let $K:I \times J \to \R$ be a kernel, and let $\phi:\widetilde{I} \to I$ and $\psi:\widetilde{J} \to J$ be continuous and strictly increasing functions. Define
\[
\widetilde{K}(x,y) \leqdef K(\phi(x),\psi(y)), \qquad x \in \widetilde{I}, \quad y \in \widetilde{J}.
\]
Then, for every $m \in \N$, $x_1 < \cdots < x_m$ in $\widetilde{I}$, and $y_1 < \cdots < y_m$ in $\widetilde{J}$,
\begin{equation}\label{eq:increasing.change.minor}
\det\!\big[\widetilde{K}(x_i,y_j)\big]_{i,j=1}^m = \det\!\big[K(\phi(x_i),\psi(y_j))\big]_{i,j=1}^m.
\end{equation}
Thus, the two determinants in \eqref{eq:increasing.change.minor} have the same sign. Consequently, $K$ is $\mathrm{TP}_n$ (respectively, $\mathrm{STP}_n$) on $\phi(\widetilde{I}) \times \psi(\widetilde{J})$ if and only if $\widetilde{K}$ is $\mathrm{TP}_n$ (respectively, $\mathrm{STP}_n$) on $\widetilde{I} \times \widetilde{J}$.
\end{lemma}

\begin{proof}
This result follows from \citet[pp.~17--18, Theorem~2.1(b)]{Karlin1968}, applied to the restriction of $K$ to $\phi(\widetilde{I}) \times \psi(\widetilde{J})$.
\end{proof}

\begin{lemma}[The continuous Cauchy--Binet formula]\label{lem:continuous.Cauchy.Binet}
Let $K_1,K_2:\R^2 \to \R$ be measurable functions, and define the composition kernel
\[
M(x,y) \leqdef \int_{\R} K_1(x,t) K_2(t,y) \, \rd t, \qquad x,y \in \R.
\]
Fix $x_1,\ldots,x_n$ and $y_1,\ldots,y_n$, and suppose that the integrals defining $M(x_i,y_j)$ and the determinant integral below converge absolutely. Then
\begin{equation}\label{eq:continuous.Cauchy.Binet}
\det\!\big[M(x_i,y_j)\big]_{i,j=1}^n = \operatornamewithlimits\idotsint_{t_1 < \cdots < t_n}\det\!\big[K_1(x_i,t_k)\big]_{i,k=1}^n \, \det\!\big[K_2(t_k,y_j)\big]_{k,j=1}^n \prod_{k=1}^n \, \rd t_k.
\end{equation}
In particular, the composition of two $\mathrm{TP}_n$ (respectively, $\mathrm{STP}_n$) kernels is $\mathrm{TP}_n$ (respectively, $\mathrm{STP}_n$) whenever the integral in \eqref{eq:continuous.Cauchy.Binet} is finite.
\end{lemma}

\begin{proof}
This is the basic composition formula of \citet[Eq.~(2.5), p.~17]{Karlin1968}.
\end{proof}

\section{Total positivity of the Gauss hypergeometric kernel}\label{sec:TP.F}

We begin by establishing our first main result, which resolves the first problem posed by Li.

\begin{theorem}\label{thm:strict.total.positivity}
The kernel $\mathcal{F}(\Delta,z) = {}_2F_1(\Delta,\Delta;2\Delta;z)$, $\Delta > 0$, $0 < z < 1$, is $\mathrm{STP}_{\infty}$.
\end{theorem}

Two further standard facts are recorded before the proof of Theorem~\ref{thm:strict.total.positivity}.

\begin{lemma}[The exponential kernel]\label{lem:exponential.kernel}
The kernel $(x,y) \mapsto e^{xy}$ is $\mathrm{STP}_{\infty}$ on $\R^2$.
\end{lemma}

\begin{proof}
We refer to \citet[Eq.~(2.1), pp.~15--16]{Karlin1968} for a proof of this result. Alternative proofs are also given by \citet[p.~119]{lehmann1986testing} and \citet{grossrichards1989}.
\end{proof}

\begin{lemma}[The indicator and truncated-power kernels]\label{lem:truncated.power.kernel}
Let $z_+ \leqdef \max\{z,0\}$ for all $z \in \R$. The kernels
\[
K_1(x,y) \leqdef \ind_{\{x<y\}}, \qquad K_2(x,y) \leqdef (y - x)_+,
\]
are $\mathrm{TP}_{\infty}$ on $\R^2$.
\end{lemma}

\begin{proof}
That the kernels $K_1(x,y)$ and $K_2(x,y)$ are $\mathrm{TP}_{\infty}$ on $\R^2$ is classical; see \citet{Karlin1968}, \citet{schoenberg1973cardinal}, and, specifically for truncated-power kernels, \citet{Bojanov1990}. Indeed, if $x_1 < \cdots < x_n$ and $y_1 < \cdots < y_n$, then
\begin{equation}\label{eq:J.kernel.explicit}
\det\!\big[K_1(x_i,y_j)\big]_{i,j=1}^n =
\begin{cases}
1, & x_1 < y_1 \leq x_2 < y_2 \leq \cdots \leq x_n < y_n, \\
0, & \textrm{otherwise};
\end{cases}
\end{equation}
see, e.g., \citet[Example~5.2]{RichardsUhler2019}. In the case of the kernel $K_2(x,y)$, it is simple to verify that
\[
K_2(x,y) = \int_{\R} K_1(x,t) K_1(t,y) \, \rd t,
\]
and then it follows by \eqref{eq:J.kernel.explicit} and Lemma~\ref{lem:continuous.Cauchy.Binet} that $K_2$ is $\mathrm{TP}_{\infty}$.
\end{proof}

\subsection{Proof of Theorem~\ref{thm:strict.total.positivity}}

For $\Delta > 0$ and $z \in (0,1)$, define
\[
\eta(z) \leqdef -\frac{1}{2}\log(1 - z) > 0, \qquad F_{\Delta}(x) \leqdef (1 + e^{-2x})^{-\Delta}, \qquad x \in \R.
\]
According to Euler's integral representation for the Gauss hypergeometric function \cite[Eqs.~15.1.2 and~15.6.1]{Daalhuis2010},
\begin{equation}\label{eq:Euler.integral}
\mathcal{F}(\Delta,z) = \frac{1}{B(\Delta,\Delta)} \int_0^1 t^{\Delta-1}(1 - t)^{\Delta-1}(1 - zt)^{-\Delta} \, \rd t.
\end{equation}
Set $t = F_1(x) = (1 + e^{-2x})^{-1}$. Then $ \, \rd t = 2t(1 - t) \, \rd x$ and, since $1 - z = e^{-2\eta(z)}$, then we have
\[
\frac{1 - t}{1 - zt} = \frac{1}{1 + (1 - z)e^{2x}} = F_1(\eta(z) - x).
\]
Since $F_{\Delta}(x) = [F_1(x)]^{\Delta}$, then it follows from \eqref{eq:Euler.integral} that
\begin{equation}\label{eq:F.logistic.integral}
\mathcal{F}(\Delta,z) = \frac{2}{B(\Delta,\Delta)} \int_{\R} F_{\Delta}(x)F_{\Delta}(\eta(z) - x) \, \rd x.
\end{equation}
The function $F_{\Delta}$ evidently is continuous and strictly increasing, with $F_{\Delta}(x) \to 0$ as $x \to -\infty$ and $F_{\Delta}(x) \to 1$ as $x \to \infty$. Therefore, $F_{\Delta}$ is the cumulative distribution function of a random variable, which we denote by $X_{\Delta}$.

Let $X_{\Delta}$ and $X_{\Delta}'$ be independent, identically distributed random variables with the distribution function $F_{\Delta}$, and denote by $\EE$ expectation with respect to their joint probability distribution. Then by Tonelli's theorem and the mutual independence of $X_{\Delta}$ and $X_{\Delta}'$,
\begin{align}\label{eq:positive.part.identity}
\int_{\R}F_{\Delta}(x)F_{\Delta}(\eta(z) - x) \, \rd x &= \EE\left[\int_{\R}\ind_{\{X_{\Delta}\leq x \leq \eta(z) - X_{\Delta}'\}} \, \rd x\right] \nonumber \\
&= \EE[(\eta(z) - X_{\Delta} - X_{\Delta}')_+] \reqdef H(\Delta,\eta(z)).
\end{align}
Equations \eqref{eq:F.logistic.integral} and \eqref{eq:positive.part.identity} yield
\[
\mathcal{F}(\Delta,z) = \frac{2}{B(\Delta,\Delta)} H(\Delta,\eta(z)).
\]
Since the factor $2/B(\Delta,\Delta)$ is strictly positive and depends on $\Delta$ only, and the mapping $z \mapsto \eta(z)$ is strictly increasing, then Lemmas~\ref{lem:scaling.functions} and~\ref{lem:increasing.transformation} show that it suffices to prove that the kernel $H(\Delta,\eta)$ is $\mathrm{STP}_{\infty}$ on $(0,\infty)^2$.

Let $P(\Delta,s)$ denote the probability density function of $X_{\Delta} + X_{\Delta}'$. Then, by \eqref{eq:positive.part.identity},
\[
H(\Delta,\eta) = \int_{\R} P(\Delta,s) \, (\eta - s)_+ \, \rd s.
\]
We will show below that $P(\Delta,s)$ is $\mathrm{STP}_{\infty}$ on $(0,\infty)\times \R$. Assuming for now that this is true, Lemmas~\ref{lem:continuous.Cauchy.Binet} and~\ref{lem:truncated.power.kernel} imply, for any $n\in \N$, $0 < \Delta_1 < \cdots < \Delta_n$ and $\eta_1 < \cdots < \eta_n$,
\[
\det\!\big[H(\Delta_i,\eta_j)\big]_{i,j=1}^n = \operatornamewithlimits\idotsint_{s_1 < \cdots < s_n} \det\!\big[P(\Delta_i,s_k)\big]_{i,k=1}^n \, \det\!\big[(\eta_j - s_k)_+\big]_{k,j=1}^n \prod_{k=1}^n \, \rd s_k \geq 0.
\]
Note that the region $\{(s_1,\ldots,s_n) \in \R^n: s_1 < \eta_1 < s_2 < \eta_2 < \cdots < s_n < \eta_n\}$ is nonempty, open, and has strictly positive Lebesgue measure. Moreover, on this region, the matrix $[(\eta_j - s_i)_+]_{i,j=1}^n$ is upper triangular with diagonal entries $\eta_i - s_i > 0$, $i = 1,\ldots,n$, so $\det[(\eta_j - s_i)_+]_{i,j=1}^n > 0$. Thus
\[
\det\!\big[H(\Delta_i,\eta_j)\big]_{i,j=1}^n > 0.
\]
Since $n$ was chosen arbitrarily, then it follows that $H(\Delta,\eta)$ is $\mathrm{STP}_{\infty}$ on $(0,\infty)\times \R$.

It remains to show that $P(\Delta,s)$ is $\mathrm{STP}_{\infty}$ on $(0,\infty)\times \R$. To this end, let $A_{\Delta} \sim \operatorname{Gamma}(\Delta)$ and $E\sim\operatorname{Gamma}(1)$ be mutually independent gamma-distributed random variables. Since
\[
\PP(F_1(X_{\Delta}) \leq t) = \PP(X_{\Delta} \leq F_1^{-1}(t)) = F_{\Delta}(F_1^{-1}(t)) = t^{\Delta}, \qquad 0 < t < 1,
\]
then it follows that $F_1(X_{\Delta}) \sim \operatorname{Beta}(\Delta,1)$. It is also well known that $A_{\Delta}/(A_{\Delta} + E)\sim \operatorname{Beta}(\Delta,1)$. Since $F_1(x)/(1 - F_1(x)) = e^{2x}$, then
\[
e^{2X_{\Delta}} = \frac{F_1(X_{\Delta})}{1 - F_1(X_{\Delta})} \stackrel{d}{=} \frac{A_{\Delta}/(A_{\Delta} + E)}{1 - (A_{\Delta}/(A_{\Delta} + E))} = \frac{A_{\Delta}}{E},
\]
where $U \stackrel{d}{=} V$ denotes equality \textit{in distribution} of two random variables $U$ and $V$, and thus
\begin{equation}\label{eq:X.beta.gamma}
2X_{\Delta} \stackrel{d}{=} \log A_{\Delta} - \log E.
\end{equation}

Consider mutually independent random variables $A_{\Delta,1}$, $A_{\Delta,2}$, $E_1$, and $E_2$ such that $A_{\Delta,1},A_{\Delta,2}\sim\operatorname{Gamma}(\Delta)$ and $E_1,E_2\sim\operatorname{Gamma}(1)$. Applying \eqref{eq:X.beta.gamma} to the pairs $(A_{\Delta,1},E_1)$ and $(A_{\Delta,2},E_2)$ yields
\begin{equation}\label{eq:X.Delta.Delta.prime}
X_{\Delta} + X_{\Delta}' \stackrel{d}{=} \frac{1}{2} (U_{\Delta} - V),
\end{equation}
where $U_{\Delta} = \log A_{\Delta,1} + \log A_{\Delta,2}$, $V = \log E_1 + \log E_2 \stackrel{d}{=} U_1$, and $U_{\Delta}$ and $V$ are independent. Let $q_{\Delta}$ denote the probability density function of $U_{\Delta}$. Then, it follows from \eqref{eq:X.Delta.Delta.prime} and the well-known convolution formula for the probability density function of the sum of independent random variables that
\[
P(\Delta,s) = 2 \int_{\R} q_{\Delta}(u) q_1(u - 2s) \, \rd u.
\]
In light of the continuous Cauchy--Binet formula (Lemma~\ref{lem:continuous.Cauchy.Binet}), along with Lemma~\ref{lem:increasing.transformation}, it suffices to show that the kernels $(\Delta,u) \mapsto q_{\Delta}(u)$ and $(x,y) \mapsto q_1(x - y)$ are $\mathrm{STP}_{\infty}$.

It is straightforward to deduce that the probability density function of $\log A_{\Delta}$ is
\[
\ell_{\Delta}(x) \leqdef \frac{1}{\Gamma(\Delta)} \exp(\Delta x - e^x), \qquad x \in \R.
\]
Again applying the convolution formula for the probability density function of the sum of independent random variables gives
\[
q_{\Delta}(u) = \int_{\R}\ell_{\Delta}(x) \ell_{\Delta}(u - x) \, \rd x = \frac{1}{[\Gamma(\Delta)]^2} \, w(u) \, e^{\Delta u},
\]
where
\[
w(u) \leqdef \int_{\R} \exp(-e^x - e^{u - x}) \rd x, \qquad u \in \R.
\]
Note that the function $w$ is strictly positive, so Lemmas~\ref{lem:scaling.functions}~and~\ref{lem:exponential.kernel} imply that the kernel $(\Delta,u) \mapsto q_{\Delta}(u)$ is $\mathrm{STP}_{\infty}$ on $(0,\infty)\times \R$.

Consider the translation kernel corresponding to $\ell_1$, viz.,
\[
\ell_1(x - y) = e^x e^{-y}\exp\left(e^x(-e^{-y})\right).
\]
Since $x \mapsto e^x$ and $y \mapsto -e^{-y}$ are strictly increasing, then it follows from Lemmas~\ref{lem:increasing.transformation}~and~\ref{lem:exponential.kernel} that the kernel $(x,y)\mapsto \exp\left(e^x(-e^{-y})\right)$ is $\mathrm{STP}_{\infty}$ on $\R^2$. By Lemma~\ref{lem:scaling.functions}, it follows that $(x,y) \mapsto \ell_1(x - y)$ is $\mathrm{STP}_{\infty}$ on $\R^2$. Therefore, by Lemma~\ref{lem:continuous.Cauchy.Binet}, the convolution kernel
\[
q_1(x - y) = \int_{\R} \ell_1(x - t) \ell_1(t - y) \, \rd t
\]
is $\mathrm{STP}_{\infty}$ on $\R^2$. This completes the proof of Theorem~\ref{thm:strict.total.positivity}.
$\qed$

\section{Total positivity of the conformal-block kernel}\label{sec:TP.G}

Since $z^{\Delta} = e^{\Delta\log z}$ and $z \mapsto \log z$ is strictly increasing, Lemma~\ref{lem:exponential.kernel} shows that $(\Delta,z) \mapsto z^{\Delta}$ is $\mathrm{STP}_{\infty}$. Nevertheless, pointwise multiplication does not preserve total positivity in general, even at order three; see \citet[Chapter~2]{Karlin1968}. Thus, Theorem~\ref{thm:strict.total.positivity} does not by itself settle the corresponding question for $G_{\Delta}(z) = z^{\Delta}\mathcal{F}(\Delta,z)$. The next theorem establishes the exact parameter threshold for $\Delta$ at or above which $G_{\Delta}(z)$ remains $\mathrm{STP}_{\infty}$.

\begin{theorem}\label{thm:conformal.block.threshold}
For $\tau > 0$, the conformal-block kernel $G_{\Delta}(z) = z^{\Delta} {}_2F_1(\Delta,\Delta;2\Delta;z)$, $\Delta \geq \tau$, $0 < z < 1$, is $\mathrm{STP}_{\infty}$ if and only if $\tau \geq 1/2$.

Moreover, for every $\tau \in (0,1/2)$ and $z_0 \in [0,1)$, there exist an integer $m \in \N$ and real numbers $\tau < \Delta_1 < \cdots < \Delta_{2m + 1} < 1/2$ and $z_0 < z_1 < \cdots < z_{2m + 1} < 1$ such that
\begin{equation}\label{eq:sharp.det.neg}
\det\!\big[G_{\Delta_i}(z_j)\big]_{i,j=1}^{2m + 1} < 0.
\end{equation}
Consequently, there exists no $z_0 \in [0,1)$ such that the kernel $(\Delta,z) \mapsto G_{\Delta}(z)$ is $\mathrm{TP}_{\infty}$ on $(0,\infty) \times (z_0,1)$.
\end{theorem}

The proof has two parts. Subsection~\ref{subsec:above.threshold} proves the sufficiency assertion for $\Delta \geq 1/2$. A change of variables rewrites $G_{\Delta}(z)$ in terms of a family of strictly positive Legendre functions that solve the same second-order differential equation. Their behavior at infinity determines the signs of all their Wronskians, and a standard Wronskian criterion then gives the signs of the determinants evaluated at ordered $x$-values.

Subsection~\ref{subsec:threshold.sharpness} proves that the threshold is sharp. The kernel is expanded as $z \to 1$, and the leading term of each odd determinant of mixed derivatives is governed by a finite matrix formed from derivatives of the digamma function. A representation by a nonnegative measure shows that these matrices cannot all have the signs required by total positivity, and a Taylor expansion converts the resulting determinant of mixed derivatives into a strictly negative ordinary minor. The auxiliary Wronskian and determinant calculations are proved in Appendices~\ref{app:Wronskian.tools} and~\ref{app:boundary.calculations}.

\subsection{Proof of the sufficiency portion of Theorem~\ref{thm:conformal.block.threshold}}\label{subsec:above.threshold}

We first rewrite the conformal block $G_{\Delta}(z)$ in a form in which the transformed parameter $s = \Delta - 1/2$ appears as the decay rate of a solution to a differential equation. Multiplying Euler's integral formula, \eqref{eq:Euler.integral}, by $z^{\Delta}$ and then substituting $u = 2t - 1$, we obtain
\begin{equation}\label{eq:G.integral.before.Legendre}
G_{\Delta}(z) = z^{\Delta} \mathcal{F}(\Delta,z) = \frac{2^{1 - \Delta}}{B(\Delta,\Delta)}\int_{-1}^1 (1 - u^2)^{\Delta - 1}(2z^{-1} - 1 - u)^{-\Delta} \, \rd u.
\end{equation}
The Legendre function of the second kind, with the standard normalization (\citet[\S14.3(ii), Eq.~14.3.7]{Dunster2010}), has the integral representation
\begin{equation}\label{eq:Legendre.integral}
Q_{\nu}(x) = 2^{-\nu - 1}\int_{-1}^1 (1 - u^2)^{\nu}(x - u)^{-\nu - 1} \, \rd u, \qquad \nu > -1, \quad x > 1;
\end{equation}
see also \citet[\S3.7, Eq.~(5), p.~155]{ErdelyiEtAl1953}. Taking $\nu = \Delta - 1$ and $x = 2z^{-1} - 1$ in \eqref{eq:Legendre.integral} and comparing the result with \eqref{eq:G.integral.before.Legendre} yields
\[
G_{\Delta}(z) = \frac{2}{B(\Delta,\Delta)} Q_{\Delta - 1}(2z^{-1} - 1), \qquad \Delta > 0, \quad 0 < z < 1.
\]

We now make the changes of variables
\[
z = \operatorname{sech}^2(x/2), \qquad s = \Delta - \frac{1}{2}.
\]
The map $x \mapsto \operatorname{sech}^2(x/2)$ is strictly decreasing from $(0,\infty)$ onto $(0,1)$, while $\Delta \geq 1/2$ is equivalent to $s \geq 0$. Moreover,
\[
2z^{-1} - 1 = 2\cosh^2(x/2) - 1 = \cosh x.
\]
Hence,
\begin{equation}\label{eq:G.us}
G_{\Delta}(z) = \frac{2}{B(\Delta,\Delta)} \: (\sinh x)^{-1/2} \: u_s(x), \qquad \Delta > 0, \quad 0 < z < 1,
\end{equation}
where
\begin{equation}\label{eq:us.Q.function}
u_s(x) \leqdef (\sinh x)^{1/2} \: Q_{s - 1/2}(\cosh x), \qquad x > 0.
\end{equation}

The following lemma, whose proof is deferred to Appendix~\ref{app:Wronskian.tools}, establishes the differential equation satisfied by the functions $u_s$ and details their precise exponential decay behavior as $x \to \infty$.  In stating the lemma, we will encounter the strictly positive normalizing constant
\begin{equation}\label{eq:c:s:constant}
c_s \leqdef \frac{1}{2^{1/2}} \int_{-1}^1 (1 - u^2)^{s - 1/2} \, \rd u = \frac{\pi^{1/2} \Gamma(s+\frac{1}{2})}{2^{1/2} \Gamma(s+1)}, \qquad s > -\frac{1}{2}.
\end{equation}

\smallskip

\begin{lemma}[Properties of the transformed Legendre functions]\label{lem:us.properties}
For every $s \geq 0$, the function $u_s \in C^{\infty}((0,\infty))$ defined in \eqref{eq:us.Q.function} is strictly positive on $(0,\infty)$ and satisfies the differential equation
\begin{equation}\label{eq:Schrodinger.equation}
u_s''(x) + \frac{u_s(x)}{4\sinh^2 x} = s^2u_s(x).
\end{equation}
Also, for each $s > 0$, as $x \to \infty$,
\begin{equation}\label{eq:asymp.us.us.prime}
u_s(x) = c_s e^{-sx}(1 + o(1)), \qquad u_s'(x) = -s c_s e^{-sx}(1 + o(1)).
\end{equation}
Further, $u_0(x) \to c_0 \equiv \pi/2^{1/2}$ and $u_0'(x) \to 0$ as $x \to \infty$.
\end{lemma}

\smallskip

For sufficiently smooth real-valued functions $f_1,\ldots,f_n$ on a common interval, denote their Wronskian by
\[
W_x(f_1,\ldots,f_n) \leqdef \det\!\big[f_i^{(j - 1)}(x)\big]_{i,j=1}^n.
\]
The next lemma abstracts the specific differential equation of Lemma~\ref{lem:us.properties} by replacing the potential $-1/(4\sinh^2 x)$ by a general smooth potential $V(x)$.

This abstraction enables the induction carried out in the proof of the lemma. Starting from the family of strictly positive functions $(u_{s_1},\ldots,u_{s_n})$, the proof introduces the first-order operator $\mathcal{A} \leqdef \frac{\rd}{\rd x} - \smash{\frac{u_{s_1}'(x)}{u_{s_1}(x)}}$, which annihilates $u_{s_1}$, and defines the transformed functions by
\[
v_{s_i}(x) \leqdef -\mathcal{A}u_{s_i}(x), \qquad i = 2,\ldots,n.
\]
It then shows that $(v_{s_2},\ldots,v_{s_n})$ is a new family of strictly positive functions whose members have the same prescribed type of decay at infinity and satisfy differential equations of the same form, with the new common potential $\widetilde{V}(x) = V(x) - 2\frac{\rd^2}{\rd x^2}\log u_{s_1}(x)$. Consequently, this construction can be iterated to determine the sign of $W_x(u_{s_1},\ldots,u_{s_k})$ for each $k = 1,\ldots,n$.

\begin{lemma}[Wronskian sign induction]\label{lem:Darboux.Wronskians}
Let $V \in C^{\infty}((0,\infty))$, and let $0 \leq s_1 < \cdots < s_n$. For $i = 1,\ldots,n$, suppose that the strictly positive functions $u_{s_i} \in C^{\infty}((0,\infty))$ satisfy the differential equation
\[
u_{s_i}''(x) - V(x) u_{s_i}(x) = s_i^2 u_{s_i}(x), \qquad x\in (0,\infty).
\]
Suppose also that, for every $i \in \{1,\ldots,n\}$ such that $s_i > 0$, there exists a constant $a_i > 0$ such that, as $x \to \infty$,
\begin{equation}\label{eq:behavior.infinity}
u_{s_i}(x) = a_i e^{-s_ix}(1 + o(1)), \qquad u_{s_i}'(x) = -s_i a_i e^{-s_ix}(1 + o(1)).
\end{equation}
If $s_1 = 0$, suppose also that $u_{s_1}(x) \to a_1 > 0$ and $u_{s_1}'(x) \to 0$ as $x \to \infty$. Then, for $1 \leq k \leq n$ and $x > 0$,
\[
\operatorname{sgn}W_x(u_{s_1},\ldots,u_{s_k}) = (-1)^{k(k - 1)/2}.
\]
\end{lemma}

\smallskip

Once these Wronskian signs are known, the following standard criterion turns them into signs of determinants evaluated at ordered $x$-values. Its proof rescales the functions so that all their initial Wronskians are strictly positive and then applies the Wronskian criterion for positive extended complete Tchebycheff systems; see \citet[Theorem~1.1, p.~376]{KarlinStudden1966}.

\smallskip

\begin{lemma}[Passage to determinants at ordered points]\label{lem:Wronskian.to.collocation}
Let $I \subseteq \R$ be an open interval, let $f_1,\ldots,f_n:I \to \R$ be $(n - 1)$-times continuously differentiable, and suppose that $W_x(f_1,\ldots,f_k) \neq 0$ for all $x \in I$ and all $k = 1,\ldots,n$. Then, for all $x_1,\ldots,x_n \in I$ such that $x_1 < \cdots < x_n$, the determinant $\det[f_i(x_j)]_{i,j=1}^n$ is nonzero and has the same sign as $W_x(f_1,\ldots,f_n)$.
\end{lemma}

To avoid interrupting the flow of the proof of the sufficiency portion of Theorem~\ref{thm:conformal.block.threshold}, the proofs of Lemmas~\ref{lem:Darboux.Wronskians} and~\ref{lem:Wronskian.to.collocation} are also deferred to Appendix~\ref{app:Wronskian.tools}. In order to apply Lemma~\ref{lem:Darboux.Wronskians} to \eqref{eq:Schrodinger.equation} with $V(x) = -1/(4\sinh^2x)$, first we verify by \eqref{eq:us.asymptotic} and \eqref{eq:us.derivative.asymptotic} that the hypotheses \eqref{eq:behavior.infinity} of Lemma~\ref{lem:Darboux.Wronskians} are satisfied, with $a_i = c_{s_i}$ as defined in \eqref{eq:c:s:constant}. Then it follows that, for $0 \leq s_1 < \cdots < s_n$,
\[
\operatorname{sgn}W_x(u_{s_1},\ldots,u_{s_n}) = (-1)^{n(n - 1)/2}.
\]
Lemma~\ref{lem:Wronskian.to.collocation} now yields
\begin{equation}\label{eq:u.collocation.sign}
\operatorname{sgn}\det\!\big[u_{s_i}(x_j)\big]_{i,j=1}^n = (-1)^{n(n - 1)/2}, \qquad 0 < x_1 < \cdots < x_n.
\end{equation}

Let $1/2 \leq \Delta_1 < \cdots < \Delta_n$ and $0 < z_1 < \cdots < z_n < 1$, and let $s_i = \Delta_i - \tfrac{1}{2}$ and $z_j = \operatorname{sech}^2(x_j/2)$, $i,j=1,\ldots,n$. Then $x_1 > \cdots > x_n > 0$ and, by \eqref{eq:G.us},
\[
G_{\Delta_i}(z_j) = \frac{2}{B(\Delta_i,\Delta_i)} \: (\sinh x_j)^{-1/2} \: u_{s_i}(x_j).
\]
Reversing the $n$ columns places the $x$-values in strictly increasing order, and the sign contributed by this reversal is $(-1)^{n(n - 1)/2}$. Therefore, \eqref{eq:u.collocation.sign} yields
\[
\begin{aligned}
\det\!\big[G_{\Delta_i}(z_j)\big]_{i,j=1}^n
&= \left(\prod_{i=1}^n \frac{2}{B(\Delta_i,\Delta_i)}\right)\left(\prod_{j=1}^n (\sinh x_j)^{-1/2}\right) \det\!\big[u_{s_i}(x_j)\big]_{i,j=1}^n \\
&= \left(\prod_{i=1}^n \frac{2}{B(\Delta_i,\Delta_i)}\right)\left(\prod_{j=1}^n (\sinh x_j)^{-1/2}\right) (-1)^{n(n - 1)/2} \: \det\!\big[u_{s_i}(x_{n + 1 - j})\big]_{i,j=1}^n > 0.
\end{aligned}
\]
Since $n$ was chosen arbitrarily, this proves that the kernel $(\Delta,z) \mapsto G_{\Delta}(z)$ is $\mathrm{STP}_{\infty}$ on the region $[1/2,\infty) \times (0,1)$. This completes the proof of the sufficiency portion of Theorem~\ref{thm:conformal.block.threshold}.
$\qed$

\subsection{Proof of the sharpness assertion in Theorem~\ref{thm:conformal.block.threshold}}\label{subsec:threshold.sharpness}

Let $\tau \in (0,1/2)$ and $z_0 \in [0,1)$ be fixed. The goal is to construct a strictly negative minor of the conformal-block kernel $G_{\Delta}(z)$ with all its $\Delta$-values in $(\tau,1/2)$ and all its $z$-values in $(z_0,1)$.

The proof is organized in four steps. \ref{Step:1} reduces the construction of a strictly negative ordinary minor of the function $q(\Delta,y)$ in \eqref{eq:q.definition} to the construction of a strictly negative determinant of mixed partial derivatives of $q(\Delta,y)$. \ref{Step:2} reduces the large-$y$ sign of the determinant of those mixed partial derivatives to the sign of a matrix formed from derivatives of the digamma function. \ref{Step:3} derives both a parameter value $\Delta_0 \in (\tau,1/2)$ and an order for which the corresponding signed determinant is strictly negative. Finally, \ref{Step:4} separates the coincident row and column parameters and returns to the original conformal-block kernel.

To ensure that the basic strategy of the proof is not obscured by the technical details, the proofs of the central lemmas used in these steps (Lemmas~\ref{lem:confluent.Taylor}, \ref{lem:q.boundary.expansion}, \ref{lem:confluent.boundary.asymptotic}, and \ref{lem:H.sign.obstruction}), alongside other auxiliary details, are deferred to Appendix~\ref{app:boundary.calculations}. This also isolates the intricate Wronskian and determinant calculations, thereby maintaining focus on the core argument.

\bigskip
\noindent
{\bf \setword{Step~1}{Step:1}: Reduction to a determinant of mixed derivatives.}

\smallskip

Define the function
\begin{equation}\label{eq:q.definition}
q(\Delta,y) \leqdef B(\Delta,\Delta)G_{\Delta}(1 - e^{-y}), \qquad \Delta > 0, \quad y > 0.
\end{equation}
By Lemmas~\ref{lem:scaling.functions} and~\ref{lem:increasing.transformation}, the strict positivity of $B(\Delta,\Delta)$ and the strictly increasing map $y \mapsto 1 - e^{-y}$ preserve the signs of the kernel determinants. Lemma~\ref{lem:q.boundary.expansion} below also shows that $q(\Delta,y)$ is infinitely differentiable on $(0,\infty) \times (0,\infty)$.

Henceforth, we denote the partial differential operators $\partial/\partial\Delta$ and $\partial/\partial y$ by $\partial_{\Delta}$ and $\partial_y$, respectively. For $N \in \N$, define the determinant of mixed partial derivatives:
\[
D_N(\Delta,y) \leqdef \det\!\left[\partial_{\Delta}^i\partial_y^j q(\Delta,y)\right]_{i,j=0}^{N - 1}.
\]
Fix $x_1 < \cdots < x_N$ and $w_1 < \cdots < w_N$, and define
\begin{equation}\label{eq:M.V}
M \leqdef N(N - 1)/2, \qquad V(t_1,\ldots,t_N) \leqdef \prod_{1\leq i<j\leq N}(t_j - t_i).
\end{equation}
By Lemma~\ref{lem:confluent.Taylor},
\[
\lim_{\varepsilon \to 0^{+}, \, \eta \to 0^{+}} \frac{\det\!\big[q(\Delta + \varepsilon x_r,y + \eta w_s)\big]_{r,s=1}^N}{\varepsilon^M\eta^M V(x_1,\ldots,x_N)V(w_1,\ldots,w_N)} = \frac{D_N(\Delta,y)}{\big(\prod_{j=0}^{N - 1}j!\big)^2}.
\]
Each factor in the denominator on the left-hand side is strictly positive, as is the factorial factor on the right. The limit therefore has the same sign as $D_N(\Delta,y)$. If $D_N(\Delta,y) < 0$, then $\det[q(\Delta + \varepsilon x_r,y + \eta w_s)]_{r,s=1}^N < 0$ for all sufficiently small strictly positive $\varepsilon$ and $\eta$. The remaining task is consequently to find an odd integer $N$, a point $\Delta \in (\tau,1/2)$, and arbitrarily large values of $y$ for which $D_N(\Delta,y) < 0$.

\bigskip
\noindent
{\bf \setword{Step~2}{Step:2}: Large-$\boldsymbol{y}$ asymptotics of the determinant of mixed derivatives.}

\smallskip

The relation $z = 1 - e^{-y}$ shows that $z \uparrow 1$ is equivalent to $y \to \infty$. Define
\begin{equation}\label{eq:h.definition}
\lambda \leqdef \Delta(\Delta - 1), \qquad h(\Delta) \leqdef -2\gamma - 2\psi(\Delta).
\end{equation}
The following lemma provides a convergent series representation for $q(\Delta,y)$. In particular, for every $(i,j)$ such that $0 \leq i,j \leq N - 1$, the $(i,j)$th entry of the matrix defining $D_N(\Delta,y)$ is obtained by applying the mixed derivative $\partial_{\Delta}^i\partial_y^j$ term by term to this series.

\begin{lemma}[Series representation of $q$ near $z = 1$]\label{lem:q.boundary.expansion}
There exist sequences of real polynomials $(P_k)_{k \in \N_0}$ and $(S_k)_{k \in \N_0}$ such that $S_0 = 0$, $P_k$ has degree $k$ and leading coefficient $(k!)^{-2}$ for every $k \in \N_0$, $S_k$ has degree at most $k$ for every $k \geq 1$, and
\begin{equation}\label{eq:q.mode.expansion}
q(\Delta,y) = \sum_{k=0}^{\infty} e^{-ky} [P_k(\lambda)y + Q_k(\Delta)], \qquad \Delta > 0, \quad y > 0,
\end{equation}
where, for $k \in \N_0$,
\[
Q_k(\Delta) \leqdef P_k(\lambda) h(\Delta) + S_k(\lambda).
\]
For every compact set $\mathcal{C} \subseteq (0,\infty)$, every $y_{\star} > 0$, and every fixed $a,b \in \N_0$, the series obtained by applying $\partial_{\Delta}^a\partial_y^b$ to the right-hand side term by term converges uniformly for $\Delta \in \mathcal{C}$ and $y \geq y_{\star}$ and equals $\partial_{\Delta}^a\partial_y^b q(\Delta,y)$. Moreover, for every $R \in (0,1)$ and every fixed $a \in \N_0$, the series
\[
\sum_{k=0}^{\infty} \partial_{\Delta}^a P_k(\lambda)t^k, \qquad \sum_{k=0}^{\infty}\partial_{\Delta}^a Q_k(\Delta)t^k
\]
both converge absolutely and uniformly for $\Delta \in \mathcal{C}$ and $t \in \mathbb{C}$ with $|t| \leq R$.
\end{lemma}

For $0 < \Delta < 1/2$, define
\begin{equation}\label{eq:xi.c.definition}
\xi \leqdef \Big(\frac{1}{2} - \Delta\Big)^2, \qquad c(\xi) \leqdef -2\gamma - 2\psi\Big(\frac{1}{2} - \xi^{1/2}\Big),
\end{equation}
and, for $m \in \N$, define the matrix
\begin{equation}\label{eq:H.m}
H_m(\xi) \leqdef \left[\frac{c^{(i + j)}(\xi)}{(i + j)!}\right]_{i,j=1}^m.
\end{equation}
The following lemma shows that the behavior of the dominant large-$y$ asymptotic term of $D_{2m+1}$ is governed by $\det H_m(\xi)$.

\begin{lemma}[Asymptotic reduction of $D_{2m+1}$ to $\det H_m(\xi)$]\label{lem:confluent.boundary.asymptotic}
Let $m \in \N$ and $0 < \Delta < 1/2$. Then, locally uniformly in $\Delta$ as $y \to \infty$,
\begin{equation}\label{eq:D.asymptotic}
D_{2m + 1}(\Delta,y) = C_m e^{-m^2y} \big[(2\Delta - 1)^{m(2m + 1)}y\det H_m(\xi) + O(1)\big],
\end{equation}
where $C_m > 0$ is a constant depending only on $m$.
\end{lemma}

If $\det H_m(\xi) \neq 0$ and $0 < \Delta < 1/2$, then Lemma~\ref{lem:confluent.boundary.asymptotic} shows that, for all sufficiently large $y$,
\begin{equation}\label{eq:sign.relations}
\operatorname{sgn}\big(D_{2m + 1}(\Delta,y)\big) = \operatorname{sgn}\big((2\Delta - 1)^{m(2m + 1)}\det H_m(\xi)\big) = \operatorname{sgn}\big((-1)^m\det H_m(\xi)\big).
\end{equation}
Hence, $D_{2m + 1}(\Delta_0,y)$ is strictly negative for all sufficiently large $y$ as soon as a point $\Delta_0 \in (\tau,1/2)$ and an integer $m \in \N$ are found such that $(-1)^m\det H_m((\tfrac{1}{2} - \Delta_0)^2) < 0$.

\bigskip
\noindent
{\bf \setword{Step~3}{Step:3}: Choice of $\boldsymbol{\Delta_0}$ and the order of the determinant.}

\smallskip

Since a suitable $m$ depends on the chosen point $\Delta_0$, the order $2m + 1$ cannot be fixed in advance. The next lemma determines both $\Delta_0$ and $m$ simultaneously.

\begin{lemma}[Existence of a strictly negative value for $(-1)^m \det H_m(\xi)$]\label{lem:H.sign.obstruction}
For any nonempty open interval $I$ with closure in $(0,1/2)$, there exist $\Delta_0 \in I$ and $m \in \N$ such that
\begin{equation}\label{eq:H.negative.sign}
(-1)^m\det H_m(\xi_0) < 0,
\end{equation}
where $\xi_0 \leqdef (\tfrac{1}{2} - \Delta_0)^2$.
\end{lemma}

Choose a nonempty open interval $I$ whose closure is contained in $(\tau,1/2)$. By Lemma~\ref{lem:H.sign.obstruction}, there exist $\Delta_0 \in I$ and $m \in \N$ for which \eqref{eq:H.negative.sign} holds. Lemma~\ref{lem:H.sign.obstruction} and \eqref{eq:sign.relations} therefore imply
\begin{equation}\label{eq:negative.confluent.minor}
D_{2m + 1}(\Delta_0,y) < 0
\end{equation}
for all sufficiently large $y$.

\newpage
\noindent
{\bf \setword{Step~4}{Step:4}: Construction of a strictly negative conformal-block minor.}

\smallskip

By \eqref{eq:negative.confluent.minor}, we may choose $y_0$ so large that
\[
D_{2m + 1}(\Delta_0,y_0) < 0 \quad \text{and} \quad y_0 > -\log(1 - z_0).
\]
Let $N = 2m + 1$. It remains to use the inequality $D_N(\Delta_0,y_0) < 0$ to choose distinct row parameters $\Delta_1 < \cdots < \Delta_N$ and column parameters $y_1 < \cdots < y_N$ for which the kernel determinant $\det[q(\Delta_i,y_j)]_{i,j=1}^N$ is strictly negative. Fix real numbers $x_1 < \cdots < x_N$ and $w_1 < \cdots < w_N$. With $M$ and $V$ defined in \ref{Step:1}, Lemma~\ref{lem:confluent.Taylor}, applied to $q$, gives
\[
\det\!\big[q(\Delta_0 + \varepsilon x_r,y_0 + \eta w_s)\big]_{r,s=1}^N
= \frac{\varepsilon^M \eta^M V(x_1,\ldots,x_N) V(w_1,\ldots,w_N)}{\big(\prod_{j=0}^{N - 1}j!\big)^2}D_N(\Delta_0,y_0) + o(\varepsilon^M \eta^M).
\]
Since the ordered sequences yield strictly positive Vandermonde determinants and $D_N(\Delta_0,y_0) < 0$, then it follows that
\[
\det\!\big[q(\Delta_0 + \varepsilon x_r,y_0 + \eta w_s)\big]_{r,s=1}^N < 0
\]
for all sufficiently small strictly positive $\varepsilon$ and $\eta$. Now take $\varepsilon$ small enough that the strictly increasing values $\Delta_i \leqdef \Delta_0 + \varepsilon x_i$, $1 \leq i \leq N$, all belong to $I$. Also take $\eta$ small enough that the strictly increasing values $y_j \leqdef y_0 + \eta w_j$, $1 \leq j \leq N$, are all larger than $-\log(1 - z_0)$. Then the values $z_j \leqdef 1 - e^{-y_j}$ satisfy $z_0 < z_1 < \cdots < z_N < 1$. The inequality \eqref{eq:negative.confluent.minor} permits $y_0$ to be chosen arbitrarily large; after each such choice, $\eta$ can be chosen small enough to keep every $y_j$ close to $y_0$. Hence, all the values $z_j = 1 - e^{-y_j}$ can be made arbitrarily close to one. Finally, since $B(\Delta_i,\Delta_i) > 0$, \eqref{eq:q.definition} yields
\[
\det\!\big[G_{\Delta_i}(z_j)\big]_{i,j=1}^{2m + 1} = \left(\prod_{i=1}^{2m + 1} \frac{1}{B(\Delta_i,\Delta_i)}\right)\det\!\big[q(\Delta_i,y_j)\big]_{i,j=1}^{2m + 1} < 0,
\]
with every $\Delta_i \in I \subseteq (\tau,1/2)$ and every $z_j \in (z_0,1)$, which proves \eqref{eq:sharp.det.neg}. This completes the proof of the sharpness assertion in Theorem~\ref{thm:conformal.block.threshold}.
$\qed$

\appendix

\begin{appendices}

\section{Proofs of the lemmas used in the proof of the sufficiency portion of Theorem~\ref{thm:conformal.block.threshold}}\label{app:Wronskian.tools}

\subsection{Proof of Lemma~\ref{lem:us.properties}}

Let $s\geq 0$. It follows from the Legendre differential equation (\citet[Eq.~14.2.1]{Dunster2010}) that the function $Q_{s - \frac{1}{2}}(y)$ satisfies the equation
\[
(1 - y^2)Q_{s - \frac{1}{2}}''(y) - 2y Q_{s - \frac{1}{2}}'(y) + \left(s^2 - \frac{1}{4}\right)Q_{s - \frac{1}{2}}(y) = 0, \qquad y > 1.
\]
Defining $\widetilde{Q}_s(x) \leqdef Q_{s - \frac{1}{2}}(\cosh x)$, we have
\[
\begin{aligned}
\widetilde{Q}_s'(x) &= Q_{s - \frac{1}{2}}'(\cosh x)\sinh x, \\
\widetilde{Q}_s''(x) &= Q_{s - \frac{1}{2}}''(\cosh x)\sinh^2 x + Q_{s - \frac{1}{2}}'(\cosh x)\cosh x,
\end{aligned}
\]
which implies that $\widetilde{Q}_s$ satisfies
\begin{equation}\label{eq:vs.Legendre.equation}
\widetilde{Q}_s''(x) + \frac{\cosh x}{\sinh x}\widetilde{Q}_s'(x) - \left(s^2 - \frac{1}{4}\right)\widetilde{Q}_s(x) = 0, \qquad x > 0.
\end{equation}
To see explicitly what the factor $(\sinh x)^{1/2}$ accomplishes in \eqref{eq:us.Q.function}, write $\widetilde{Q}_s(x) = (\sinh x)^{-1/2}u_s(x)$. Direct differentiation gives
\[
\begin{aligned}
\widetilde{Q}_s'(x) &= (\sinh x)^{-1/2}\left\{u_s'(x) - \frac{1}{2}\coth(x)u_s(x)\right\}, \\
\widetilde{Q}_s''(x) &= (\sinh x)^{-1/2}\left\{u_s''(x) - \coth(x)u_s'(x) + \left(\frac{1}{2}\operatorname{csch}^2(x) + \frac{1}{4}\coth^2(x)\right)u_s(x)\right\}.
\end{aligned}
\]
Substitution in \eqref{eq:vs.Legendre.equation}, followed by the identity $\coth^2(x) = 1 + \operatorname{csch}^2(x)$, gives
\[
u_s''(x) + \frac{1}{4\sinh^2x}u_s(x) = s^2u_s(x), \qquad x > 0,
\]
which proves \eqref{eq:Schrodinger.equation}.

The integral formula \eqref{eq:Legendre.integral} also gives
\begin{equation}\label{eq:us.function.integral}
u_s(x) = (\sinh x)^{1/2}2^{-s - \frac{1}{2}}\int_{-1}^1 (1 - u^2)^{s - \frac{1}{2}}(\cosh x - u)^{-s - \frac{1}{2}} \, \rd u.
\end{equation}
Every factor in the integrand is strictly positive for $-1 < u < 1$, and any endpoint singularities are integrable for $s > -1/2$. Therefore, $u_s(x) > 0$ throughout that range.

For the behavior as $x \to \infty$, factor $(\cosh x)^{-s - \frac{1}{2}}$ from the integrand and write
\[
u_s(x) = \frac{(\sinh x)^{1/2}}{2^{s + \frac{1}{2}}(\cosh x)^{s + \frac{1}{2}}} I_s(x),
\]
where
\[
I_s(x) \leqdef \int_{-1}^1 (1 - u^2)^{s - 1/2}\left(1 - \frac{u}{\cosh x}\right)^{-s - \frac{1}{2}} \, \rd u.
\]
For every fixed $s \geq 0$, the factor $u\mapsto 1 - u/\cosh x$ is uniformly bounded below by $1/2$ for all sufficiently large $x$ and converges uniformly to one for $-1 \leq u \leq 1$, as $x\to \infty$. Consequently, by the dominated convergence theorem,
\[
I_s(x) \to \int_{-1}^1 (1 - u^2)^{s - \frac{1}{2}} \, \rd u, \qquad x\to \infty.
\]
It is elementary that, as $x\to \infty$,
\[
(\sinh x)^{1/2} = 2^{-1/2}e^{x/2}(1 + O(e^{-2x})), \qquad (\cosh x)^{-s - \frac{1}{2}} = 2^{s + \frac{1}{2}} e^{-(s + \frac{1}{2})x}(1 + O(e^{-2x})),
\]
and therefore we obtain, for $s \geq 0$,
\begin{equation}\label{eq:us.asymptotic}
u_s(x) = c_s e^{-sx}(1 + o(1)),
\end{equation}
where $c_s$ is defined in \eqref{eq:c:s:constant}.  This proves the left-hand side of \eqref{eq:asymp.us.us.prime} and also that $u_0(x) \to c_0 > 0$ as $x\to\infty$.

Differentiation of \eqref{eq:us.function.integral} under the integral sign is justified by the same uniform bound, and it yields
\[
\frac{u_s'(x)}{u_s(x)} = \frac{1}{2}\coth x - \left(s + \frac{1}{2}\right)\sinh x\frac{\int_{-1}^1 (1 - v^2)^{s - \frac{1}{2}}(\cosh x - v)^{-s - 3/2} \, \rd v}{\int_{-1}^1 (1 - v^2)^{s - \frac{1}{2}}(\cosh x - v)^{-s - \frac{1}{2}} \, \rd v}.
\]
Factoring the powers $(\cosh x)^{-s - 3/2}$ and $(\cosh x)^{-s - \frac{1}{2}}$ from the numerator and denominator, respectively, shows that the quotient of the two integrals is $(\cosh x)^{-1}(1 + o(1))$. Therefore,
\[
\frac{u_s'(x)}{u_s(x)} = \frac{1}{2} - \left(s + \frac{1}{2}\right) + o(1) = -s + o(1), \qquad x\to \infty.
\]
Together with \eqref{eq:us.asymptotic}, this yields
\begin{equation}\label{eq:us.derivative.asymptotic}
u_s'(x) = -s c_s e^{-sx}(1 + o(1)), \qquad s > 0,
\end{equation}
which proves the right-hand side of \eqref{eq:asymp.us.us.prime}. Also, taking $s = 0$ in the logarithmic-derivative asymptotic above and using $u_0(x) \to c_0 > 0$ gives $u_0'(x) \to 0$. This concludes the proof of Lemma~\ref{lem:us.properties}.
$\qed$

\subsection{Proof of Lemma~\ref{lem:Darboux.Wronskians}}

We proceed by induction on $k$. The assertion for $k=1$ follows immediately from
\[
W_x(u_{s_1}) = u_{s_1}(x) > 0.
\]
Suppose that $k \geq 2$ and that the assertion has been proved for every family of $k - 1$ functions satisfying the hypotheses of the lemma. For $i = 2,\ldots,k$, direct differentiation and the differential equations give
\begin{equation}\label{eq:derivative.Wronskian}
\frac{\rd}{\rd x}W_x(u_{s_1},u_{s_i})
= \frac{\rd}{\rd x}\{u_{s_1}u_{s_i}' - u_{s_1}'u_{s_i}\}
= u_{s_1}u_{s_i}'' - u_{s_1}''u_{s_i}
= (s_i^2 - s_1^2)u_{s_1}(x)u_{s_i}(x) > 0.
\end{equation}
By the assumed behavior of the functions $u_{s_i}$ at infinity, we obtain $W_x(u_{s_1},u_{s_i}) \to 0$ as $x\to \infty$, and this holds also if $s_1 = 0$. Integrating \eqref{eq:derivative.Wronskian} over the interval $(x,\infty)$ therefore yields
\begin{equation}\label{eq:Wronskian.integral}
W_x(u_{s_1},u_{s_i}) = -(s_i^2 - s_1^2)\int_x^{\infty}u_{s_1}(t)u_{s_i}(t) \, \rd t < 0.
\end{equation}
Since $s_1 \geq 0$, then
\[
\int_x^{\infty}u_{s_1}(t)u_{s_i}(t) \, \rd t = \frac{a_1a_i}{s_1 + s_i}e^{-(s_1 + s_i)x}(1 + o(1)),
\]
as $x \to \infty$, and by substitution in \eqref{eq:Wronskian.integral} we obtain
\begin{equation}\label{eq:Wronskian.asymptotic}
W_x(u_{s_1},u_{s_i}) = -(s_i - s_1)u_{s_1}(x)u_{s_i}(x)(1 + o(1)).
\end{equation}

To complete the induction step, the Wronskian of the $k$ original functions $u_{s_1},\ldots,u_{s_k}$ must first be reduced to a Wronskian of $k - 1$ transformed functions. It will then remain to verify that the transformed functions are strictly positive with the prescribed behavior at infinity, and that they satisfy differential equations of the same form with a new common smooth potential.

Consider the columns of the Wronskian matrix. To reduce its determinant by expansion along the first row, the last $k - 1$ entries of that row should be transformed into zeros. This suggests using a monic first-order differential operator that annihilates $u_{s_1}$. To this end, let
\[
\mathcal{A} \leqdef \frac{\rd}{\rd x} - w(x), \qquad w(x) \leqdef \frac{u_{s_1}'(x)}{u_{s_1}(x)}.
\]
Since $u_{s_1}$ is strictly positive and smooth, then the function $w$ also is smooth, and $\mathcal{A} u_{s_1} = 0$.

Now fix $x > 0$, and let $\bb{C}_j$ denote the $j$th column of the Wronskian matrix $[u_{s_i}^{(j - 1)}(x)]_{i,j=1}^k$. For $j = k,\ldots,2$, in decreasing order, replace $\bb{C}_j$ by
\[
\bb{C}_j - \sum_{\ell=0}^{j - 2}\binom{j - 2}{\ell}w^{(j - 2 - \ell)}(x)\bb{C}_{\ell + 1}.
\]
The columns occurring in the sum have not yet been changed, and each operation adds a linear combination of other columns to $\bb{C}_j$, so the determinant is preserved. By the Leibniz rule, the $i$th entry of the new $j$th column is
\[
u_{s_i}^{(j - 1)}(x) - \sum_{\ell=0}^{j - 2}\binom{j - 2}{\ell}w^{(j - 2 - \ell)}(x)u_{s_i}^{(\ell)}(x) = \frac{\rd^{j-2}}{\rd x^{j-2}} [u'_{s_i}(x) - w(x)u_{s_i}(x)] = (\mathcal{A} u_{s_i})^{(j - 2)}(x).
\]
Since $\mathcal{A} u_{s_1} = 0$, the transformed matrix has the form
\[
\begin{bmatrix}
u_{s_1}(x) & 0 & \cdots & 0 \\
u_{s_2}(x) & (\mathcal{A} u_{s_2})(x) & \cdots & (\mathcal{A} u_{s_2})^{(k - 2)}(x) \\
\vdots & \vdots & & \vdots \\
u_{s_k}(x) & (\mathcal{A} u_{s_k})(x) & \cdots & (\mathcal{A} u_{s_k})^{(k - 2)}(x)
\end{bmatrix}.
\]
Expansion along the first row now produces the key identity
\begin{equation}\label{eq:Wronskian.transformed}
W_x(u_{s_1},\ldots,u_{s_k}) = u_{s_1}(x) \, W_x(\mathcal{A} u_{s_2},\ldots,\mathcal{A} u_{s_k}).
\end{equation}

We first verify strict positivity and the required behavior at infinity for the transformed functions, which are the negatives of those appearing in the Wronskian on the right-hand side of \eqref{eq:Wronskian.transformed}. For $i = 2,\ldots,k$, define
\begin{equation}\label{eq:v.si.function}
v_{s_i}(x) \leqdef -\mathcal{A} u_{s_i}(x) = -\frac{W_x(u_{s_1},u_{s_i})}{u_{s_1}(x)} > 0,
\end{equation}
where the strict positivity follows from \eqref{eq:Wronskian.integral}. Dividing \eqref{eq:Wronskian.asymptotic} by $u_{s_1}(x)$ gives
\begin{equation}\label{eq:v.si.asymptotic}
v_{s_i}(x) = (s_i - s_1)a_i e^{-s_ix}(1 + o(1)), \qquad x \to \infty.
\end{equation}
Differentiating \eqref{eq:v.si.function} and using \eqref{eq:derivative.Wronskian} gives
\[
v_{s_i}'(x) = -(s_i^2 - s_1^2)u_{s_i}(x) - \frac{u_{s_1}'(x)}{u_{s_1}(x)}v_{s_i}(x).
\]
The hypotheses imply that $u_{s_1}'(x)/u_{s_1}(x) \to -s_1$ when $s_1 > 0$. The same conclusion holds when $s_1 = 0$, since then $u_{s_1}(x) \to a_1 > 0$ and $u_{s_1}'(x) \to 0$. Hence, \eqref{eq:behavior.infinity} and \eqref{eq:v.si.asymptotic} yield
\begin{equation}\label{eq:v.prime.si.asymptotic}
\begin{aligned}
v_{s_i}'(x)
&= \{-(s_i^2 - s_1^2) + s_1(s_i - s_1)\}a_i e^{-s_ix} + o(e^{-s_ix}) \\
&= -s_i(s_i - s_1)a_i e^{-s_ix}(1 + o(1)).
\end{aligned}
\end{equation}
Thus, the transformed functions $v_{s_2},\ldots,v_{s_k}$ have the required behavior at infinity, with strictly positive constants $(s_i - s_1)a_i$.

It remains to verify that $v_{s_2},\ldots,v_{s_k}$ satisfy differential equations of the same form as those governing the $u_{s_i}$, featuring a new common smooth potential. The algebraic manipulations that follow execute a Darboux transformation \citep{Darboux1882}, a technique famously generalized to Wronskians by \citet{Crum1955}. The strategy is to factor the original second-order differential operator $-\frac{\rd^2}{\rd x^2} + V + s_1^2$ into the product $-\mathcal{B}\mathcal{A}$, where $\mathcal{A}$ is the first-order operator constructed earlier to annihilate $u_{s_1}$, and $\mathcal{B}$ is a complementary first-order operator. Reversing the order of this product naturally generates a new differential operator with a modified potential. Establishing an intertwining relation between the old and new operators will then guarantee that the remaining transformed functions automatically satisfy the new differential equation.

To execute this factorization explicitly, let
\[
\mathcal{L} \leqdef -\frac{\rd^2}{\rd x^2} + V, \qquad \mathcal{B} \leqdef \frac{\rd}{\rd x} + w.
\]
Since
\[
w' + w^2 = \frac{u_{s_1}''}{u_{s_1}} = V + s_1^2,
\]
then, by direct multiplication, we obtain
\[
\begin{aligned}
-\mathcal{B}\mathcal{A}
&= -\left(\frac{\rd}{\rd x} + w\right)\left(\frac{\rd}{\rd x} - w\right)
= -\frac{\rd^2}{\rd x^2} + w' + w^2 = \mathcal{L} + s_1^2,
\end{aligned}
\]
and
\[
\begin{aligned}
-\mathcal{A}\mathcal{B}
&= -\left(\frac{\rd}{\rd x} - w\right)\left(\frac{\rd}{\rd x} + w\right)
= -\frac{\rd^2}{\rd x^2} - w' + w^2 = -\frac{\rd^2}{\rd x^2} + (V - 2w') + s_1^2.
\end{aligned}
\]
The reversed product therefore suggests defining a new potential $\widetilde{V}$ and a new differential operator $\widetilde{\mathcal{L}}$ as follows:
\[
\widetilde{V}(x) \leqdef V(x) - 2w'(x) = V(x) - 2\frac{\rd^2}{\rd x^2}\log u_{s_1}(x), \qquad \widetilde{\mathcal{L}} \leqdef -\frac{\rd^2}{\rd x^2} + \widetilde{V}.
\]
Then
\[
-\mathcal{B}\mathcal{A} = \mathcal{L} + s_1^2, \qquad -\mathcal{A}\mathcal{B} = \widetilde{\mathcal{L}} + s_1^2.
\]
Consequently, $\mathcal{A}(\mathcal{L} + s_1^2) = -\mathcal{A}\mathcal{B}\mathcal{A} = (\widetilde{\mathcal{L}} + s_1^2)\mathcal{A}$, and we deduce the intertwining relation
\[
\mathcal{A}\mathcal{L} = \widetilde{\mathcal{L}}\mathcal{A}.
\]
Since $\mathcal{L} u_{s_i} = -s_i^2u_{s_i}$ and $v_{s_i} = -\mathcal{A} u_{s_i}$, it follows that $\widetilde{\mathcal{L}} v_{s_i} = -s_i^2v_{s_i}$, or, equivalently,
\begin{equation}\label{eq:v.si.ode}
v_{s_i}'' - \widetilde{V}v_{s_i} = s_i^2v_{s_i}.
\end{equation}
Since $u_{s_1}$ is strictly positive and smooth, the functions $\smash{\widetilde{V}}$ and $v_{s_2},\ldots,v_{s_k}$ are smooth. Therefore, \eqref{eq:v.si.asymptotic}, \eqref{eq:v.prime.si.asymptotic}, and \eqref{eq:v.si.ode} together show that the transformed functions $v_{s_2},\ldots,v_{s_k}$ satisfy all the hypotheses of the lemma with the common potential $\smash{\widetilde{V}}$, and the induction hypothesis can now be applied.

Substituting $\mathcal{A} u_{s_i} = -v_{s_i}$ in \eqref{eq:Wronskian.transformed}, we obtain
\[
W_x(u_{s_1},\ldots,u_{s_k})
= u_{s_1}(x)W_x(-v_{s_2},\ldots,-v_{s_k})
= (-1)^{k - 1}u_{s_1}(x)W_x(v_{s_2},\ldots,v_{s_k}).
\]
By the induction hypothesis, $\operatorname{sgn}W_x(v_{s_2},\ldots,v_{s_k}) = (-1)^{(k - 1)(k - 2)/2}$. Since $u_{s_1}(x) > 0$, it follows that
\[
\operatorname{sgn}W_x(u_{s_1},\ldots,u_{s_k}) = (-1)^{k - 1}(-1)^{(k - 1)(k - 2)/2} = (-1)^{k(k - 1)/2},
\]
which completes the induction. This concludes the proof of Lemma~\ref{lem:Darboux.Wronskians}.
$\qed$

\subsection{Proof of Lemma~\ref{lem:Wronskian.to.collocation}}

For each $k = 1,\ldots,n$, the function $x \mapsto W_x(f_1,\ldots,f_k)$ is continuous and nonzero on the interval $I$, and hence has constant sign. Define $\sigma_0 \leqdef 1$ and, for $k \geq 1$,
\[
\sigma_k \leqdef \operatorname{sgn}W_x(f_1,\ldots,f_k).
\]
For $j = 1,\ldots,n$, define $g_j \leqdef \sigma_{j - 1}\sigma_j f_j$. Then, for any $k = 1,\ldots,n$,
\[
W_x(g_1,\ldots,g_k)
= \left(\prod_{j=1}^k\sigma_{j - 1}\sigma_j\right)W_x(f_1,\ldots,f_k)
= \sigma_k W_x(f_1,\ldots,f_k) > 0.
\]
The second equality holds because $\prod_{j=1}^k\sigma_{j - 1}\sigma_j = \sigma_0\sigma_k\prod_{j=1}^{k - 1}\sigma_j^2 = \sigma_k$.

Let $[a,b] \subseteq I$ be a compact interval such that $a < x_1 < \cdots < x_n < b$. Since these Wronskians are strictly positive, then $(g_1,\ldots,g_n)$ is a positive extended complete Tchebycheff system on $[a,b]$; see \citet[Theorem~1.1, p.~376]{KarlinStudden1966}. Therefore, by the defining determinant property of such a system,
\[
\det\!\big[g_i(x_j)\big]_{i,j=1}^n > 0.
\]
Finally,
\[
\det\!\big[g_i(x_j)\big]_{i,j=1}^n
= \left(\prod_{i=1}^n\sigma_{i - 1}\sigma_i\right)\det\!\big[f_i(x_j)\big]_{i,j=1}^n
= \sigma_n\det\!\big[f_i(x_j)\big]_{i,j=1}^n.
\]
Since the determinant on the left is strictly positive, then $\det[f_i(x_j)]_{i,j=1}^n$ is nonzero and has sign $\sigma_n$, which is the sign of $W_x(f_1,\ldots,f_n)$. This concludes the proof of Lemma~\ref{lem:Wronskian.to.collocation}.
$\qed$

\section{Auxiliary arguments for the proof of the sharpness assertion in Theorem~\ref{thm:conformal.block.threshold}}\label{app:boundary.calculations}

This appendix contains the precise auxiliary results used in the four steps of Subsection~\ref{subsec:threshold.sharpness}. Lemma~\ref{lem:confluent.Taylor} justifies the passage between an ordinary minor and a determinant of mixed derivatives in \ref{Step:1} and \ref{Step:4}. Lemma~\ref{lem:q.boundary.expansion} gives the series used in \ref{Step:2}. Its large-$y$ consequence is Lemma~\ref{lem:confluent.boundary.asymptotic}, whose proof uses Lemmas~\ref{lem:mode.Wronskian} and~\ref{lem:interlaced.Wronskian}. The sign-selection result used in \ref{Step:3} is Lemma~\ref{lem:H.sign.obstruction}, which is proved from Lemmas~\ref{lem:H.nonzero} and~\ref{lem:c.derivative.measure}.

\subsection{From a determinant of mixed \texorpdfstring{$\Delta$- and $y$-derivatives}{Delta- and y-derivatives} to an ordinary minor}

For $N \geq 2$, when the $\Delta$-values in an ordinary kernel determinant approach one another, its rows approach one another and the determinant tends to zero. The same happens to its columns when the $y$-values approach one another. The two Vandermonde factors in the next lemma exactly capture the polynomial rates at which these determinants approach zero. After division by these factors, the limit is the determinant of mixed derivatives divided by the two products of factorials. For any fixed ordered offsets in the lemma, all the normalizing factors are strictly positive, so a nonzero mixed-derivative determinant has the same sign as the displayed ordinary minor for all sufficiently small strictly positive $\varepsilon$ and $\eta$.

\begin{lemma}[Confluent limit of a kernel minor]\label{lem:confluent.Taylor}
Let $N \in \N$, let $\mathcal{K}$ be of class $C^{2N - 2}$ near $(\Delta_0,y_0)$, and let $x_1 < \cdots < x_N$ and $w_1 < \cdots < w_N$. Recall the definition of $M$ and $V$ from \eqref{eq:M.V}:
\[
M = N(N - 1)/2, \qquad V(t_1,\ldots,t_N) = \prod_{1\leq i<j\leq N}(t_j - t_i).
\]
Then, as $\varepsilon,\eta \to 0^{+}$,
\begin{multline}
\det\!\big[\mathcal{K}(\Delta_0 + \varepsilon x_r,y_0 + \eta w_s)\big]_{r,s=1}^N \\
= \frac{\varepsilon^M\eta^M V(x_1,\ldots,x_N)V(w_1,\ldots,w_N)}{\big(\prod_{j=0}^{N - 1}j!\big)^2}\det\!\big[\partial_{\Delta}^i\partial_y^j\mathcal{K}(\Delta_0,y_0)\big]_{i,j=0}^{N - 1} + o(\varepsilon^M\eta^M). \label{eq:confluent.Taylor}
\end{multline}
\end{lemma}

\smallskip

This lemma is the two-variable confluent form of the standard divided-difference determinant identity; see \citet[Example~4, Eq.~(14), p.~49, Definition~37, p.~59, and Eqs.~(44)--(45), p.~61]{deBoor2005}. The proof is included for completeness and to record the precise normalization needed below.

\begin{proof}[Proof of Lemma~\ref{lem:confluent.Taylor}]
The assertion is immediate for $N = 1$, so suppose that $N \geq 2$. Let $\Delta_1 < \cdots < \Delta_N$ and $y_1 < \cdots < y_N$ tend to $\Delta_0$ and $y_0$, respectively. For a sequence of points $a_1,\ldots,a_{k + 1}$, the divided difference of order $k$ of a function $f$ with respect to the $\Delta$-variable is defined recursively by
\begin{align*}
[a_1]_{\Delta}f(\cdot,y) &\leqdef f(a_1,y), \\
[a_1,\ldots,a_{k + 1}]_{\Delta}f(\cdot,y) &\leqdef \frac{[a_2,\ldots,a_{k + 1}]_{\Delta}f(\cdot,y) - [a_1,\ldots,a_k]_{\Delta}f(\cdot,y)}{a_{k + 1} - a_1}.
\end{align*}
The analogous recursive definition applies to the divided difference $[a_1,\ldots,a_{k + 1}]_y f(\Delta,\cdot)$ taken with respect to the $y$-variable. Starting from the rows of $[\mathcal{K}(\Delta_r,y_s)]_{r,s=1}^N$, perform divided-difference elimination in the $\Delta$-variable. At stage $k = 1,\ldots,N - 1$, process $r = N,\ldots,k + 1$ in descending order and replace row $r$ by the difference between rows $r$ and $r - 1$, divided by $\Delta_r - \Delta_{r - k}$. The descending order ensures that both rows used at stage $k$ still come from stage $k - 1$. After stage $k$, row $r$, for $r \geq k + 1$, consists of the divided differences $[\Delta_{r-k},\ldots,\Delta_r]_{\Delta}\mathcal{K}(\mathord{\cdot},y_s)$, $1 \leq s \leq N$. Every difference $\Delta_j - \Delta_i$, $i < j$, occurs exactly once as a divisor, so their product is $V(\Delta_1,\ldots,\Delta_N)$.

To illustrate this elimination process, consider the case $N = 3$. We start with the initial matrix of evaluations. For brevity, we suppress the column index $y_s$ and denote the $r$-th row by its evaluation point $[\Delta_r]_{\Delta}$:
\[
\begin{bmatrix}
[\Delta_1]_{\Delta} \\
[\Delta_2]_{\Delta} \\
[\Delta_3]_{\Delta}
\end{bmatrix}.
\]
At stage $k = 1$, we process the rows from bottom to top ($r = 3$, then $r = 2$), replacing each with the divided difference from the row immediately above it. The descending order ensures that we do not overwrite row $2$ before row $3$ needs it:
\[
\xrightarrow{k=1} \quad
\begin{bmatrix}
[\Delta_1]_{\Delta} \\[1ex]
\dfrac{[\Delta_2]_{\Delta} - [\Delta_1]_{\Delta}}{\Delta_2 - \Delta_1} \\[2ex]
\dfrac{[\Delta_3]_{\Delta} - [\Delta_2]_{\Delta}}{\Delta_3 - \Delta_2}
\end{bmatrix}
=
\begin{bmatrix}
[\Delta_1]_{\Delta} \\
[\Delta_1, \Delta_2]_{\Delta} \\
[\Delta_2, \Delta_3]_{\Delta}
\end{bmatrix}.
\]
At stage $k = 2$, we process the remaining row ($r = 3$), dividing by the offset $\Delta_3 - \Delta_1$:
\[
\xrightarrow{k=2} \quad
\begin{bmatrix}
[\Delta_1]_{\Delta} \\[1ex]
[\Delta_1, \Delta_2]_{\Delta} \\[1ex]
\dfrac{[\Delta_2, \Delta_3]_{\Delta} - [\Delta_1, \Delta_2]_{\Delta}}{\Delta_3 - \Delta_1}
\end{bmatrix}
=
\begin{bmatrix}
[\Delta_1]_{\Delta} \\
[\Delta_1, \Delta_2]_{\Delta} \\
[\Delta_1, \Delta_2, \Delta_3]_{\Delta}
\end{bmatrix}.
\]
Ultimately, we have divided by each of $\Delta_2 - \Delta_1$, $\Delta_3 - \Delta_2$, and $\Delta_3 - \Delta_1$ exactly once, which produces the Vandermonde determinant $V(\Delta_1, \Delta_2, \Delta_3)$ in the denominator.

Applying the same elimination to the columns in the $y$-variable divides by $V(y_1,\ldots,y_N)$. The entry in position $(i + 1,j + 1)$ of the resulting matrix is exactly
\[
[\Delta_1,\ldots,\Delta_{i + 1}]_{\Delta} \, [y_1,\ldots,y_{j + 1}]_y \, \mathcal{K}.
\]
As the distinct evaluation points $\Delta_r$ and $y_s$ converge to $\Delta_0$ and $y_0$ respectively, this entry converges to $\partial_{\Delta}^i\partial_y^j\mathcal{K}(\Delta_0,y_0)/(i!j!)$. The largest mixed derivative needed has a total order of $i + j \leq 2N - 2$. Since we assumed that the function $\mathcal{K}$ is of class $C^{2N - 2}$, these continuous derivatives are guaranteed to exist, ensuring that the limit converges jointly as both sets of evaluation points merge to their targets. Consequently,
\begin{equation}\label{eq:divided.difference.limit}
\frac{\det[\mathcal{K}(\Delta_r,y_s)]_{r,s=1}^N}{V(\Delta_1,\ldots,\Delta_N)V(y_1,\ldots,y_N)} \longrightarrow \frac{\det[\partial_{\Delta}^i\partial_y^j\mathcal{K}(\Delta_0,y_0)]_{i,j=0}^{N - 1}}{\big(\prod_{j=0}^{N - 1}j!\big)^2}.
\end{equation}
Now let $\Delta_r = \Delta_0 + \varepsilon x_r$ and $y_s = y_0 + \eta w_s$. The Vandermonde determinants satisfy
\[
V(\Delta_1,\ldots,\Delta_N) = \varepsilon^M V(x_1,\ldots,x_N), \qquad V(y_1,\ldots,y_N) = \eta^M V(w_1,\ldots,w_N).
\]
Multiplying \eqref{eq:divided.difference.limit} by these two factors gives \eqref{eq:confluent.Taylor}.
\end{proof}

\subsection{Proof of Lemma~\ref{lem:q.boundary.expansion}}

The starting identity is the following convergent zero-balanced hypergeometric series, given in \cite[Eqs.~15.1.2 and~15.8.10]{Daalhuis2010}:
\begin{equation}\label{eq:zero.balanced.series}
B(\Delta,\Delta) \, {}_2F_1(\Delta,\Delta;2\Delta;1 - t) = \sum_{k=0}^{\infty} \frac{(\Delta)_k^2}{(k!)^2} \, \big[2\psi(k + 1) - 2\psi(\Delta + k) - \log t\big] \, t^k,
\end{equation}
where $\Delta > 0$, $0 < t < 1$. Define
\[
\begin{aligned}
\mathcal{P}_{\Delta}(t) &\leqdef (1 - t)^{\Delta} \sum_{k=0}^{\infty} \frac{(\Delta)_k^2}{(k!)^2}t^k, \\
\mathcal{Q}_{\Delta}(t) &\leqdef (1 - t)^{\Delta} \sum_{k=0}^{\infty} \frac{(\Delta)_k^2}{(k!)^2} \, \big[2\psi(k + 1) - 2\psi(\Delta + k)\big] \, t^k.
\end{aligned}
\]
Since $q(\Delta,y) = B(\Delta,\Delta)(1 - t)^{\Delta}{}_2F_1(\Delta,\Delta;2\Delta;1 - t)$ with $t = e^{-y}$, then by \eqref{eq:zero.balanced.series},
\begin{equation}\label{eq:q.expansion.1}
q(\Delta,y) = y\mathcal{P}_{\Delta}(e^{-y}) + \mathcal{Q}_{\Delta}(e^{-y}).
\end{equation}
Let $\widetilde{P}_k(\Delta)$ and $\widetilde{Q}_k(\Delta)$ be the Taylor coefficients at zero of the analytic functions $\mathcal{P}_{\Delta}(t)$ and $\mathcal{Q}_{\Delta}(t)$, respectively:
\begin{equation}\label{eq:q.expansion.2}
\widetilde{P}_k(\Delta) \leqdef \frac{\partial_t^k\mathcal{P}_{\Delta}(0)}{k!}, \qquad \widetilde{Q}_k(\Delta) \leqdef \frac{\partial_t^k\mathcal{Q}_{\Delta}(0)}{k!}.
\end{equation}
In particular, $\widetilde{P}_0(\Delta) = 1$ and, by \eqref{eq:h.definition}, $\widetilde{Q}_0(\Delta) = -2\gamma - 2\psi(\Delta) = h(\Delta)$. It also follows from \eqref{eq:q.expansion.1} and \eqref{eq:q.expansion.2} that
\begin{equation}\label{eq:q.preliminary.mode.expansion}
q(\Delta,y) = \sum_{k=0}^{\infty} e^{-ky} \big[\widetilde{P}_k(\Delta)y + \widetilde{Q}_k(\Delta)\big].
\end{equation}

For every compact set $\mathcal{C} \subseteq (0,\infty)$ and every fixed $a \in \N_0$, the gamma-ratio and psi-function asymptotics \cite[Eqs.~5.11.2 and~5.11.12]{AskeyRoy2010}, together with the polygamma expansion in \cite[Eq.~5.15.9]{AskeyRoy2010}, show that there exist constants $C_1,C_2,C_3 > 0$ such that, uniformly for $\Delta \in \mathcal{C}$,
\[
\left|\partial_{\Delta}^a\frac{(\Delta)_k^2}{(k!)^2}\right| + \left|\partial_{\Delta}^a\left[\frac{(\Delta)_k^2}{(k!)^2} \big[2\psi(k + 1) - 2\psi(\Delta + k)\big]\right]\right| \leq C_1(k + 1)^{C_2}\{\log(k + 2)\}^{C_3}.
\]
Therefore, the series defining $\mathcal{P}_{\Delta}(t)$ and $\mathcal{Q}_{\Delta}(t)$, together with every fixed number of their $\Delta$-derivatives, converge absolutely and uniformly for $\Delta \in \mathcal{C}$ and $t \in \mathbb{C}$ with $|t| \leq R$ whenever $R < 1$. Here, $(1 - t)^{\Delta}$ is defined using the analytic branch of $\log(1 - t)$ on $|t| < 1$.

Fix $0 < R < \rho < 1$. By the preceding uniform convergence, the functions $\partial_{\Delta}^a\mathcal{P}_{\Delta}(t)$ and $\partial_{\Delta}^a\mathcal{Q}_{\Delta}(t)$ are bounded uniformly for $\Delta \in \mathcal{C}$ and $|t| = \rho$. Cauchy's coefficient estimate, as stated in \citet[Proposition~IV.1, p.~246]{FlajoletSedgewick2009}, therefore shows that there exists a constant $C_{a,\rho} > 0$ such that
\[
\sup_{\Delta \in \mathcal{C}}\left\{|\partial_{\Delta}^a\widetilde{P}_k(\Delta)| + |\partial_{\Delta}^a\widetilde{Q}_k(\Delta)|\right\} \leq C_{a,\rho}\rho^{-k}, \qquad k \in \N_0.
\]
It follows that the generating series formed from these two sets of derivatives converge absolutely and uniformly for $\Delta \in \mathcal{C}$ and $t \in \mathbb{C}$ with $|t| \leq R$. Given $y_{\star} > 0$, choose $\rho \in (e^{-y_{\star}/2},1)$ and apply the preceding Cauchy estimate with this value of $\rho$. Applying $\partial_{\Delta}^a\partial_y^b$ to a term in \eqref{eq:q.preliminary.mode.expansion} and using the preceding coefficient estimate gives
\[
\left|\partial_{\Delta}^a \partial_y^b \left[e^{-ky} \big[\widetilde{P}_k(\Delta)y + \widetilde{Q}_k(\Delta)\big]\right]\right| \leq C(k + 1)^b\rho^{-k}(y + 1)e^{-ky},
\]
for $\Delta \in \mathcal{C}$, $y \geq y_{\star}$, and $k \geq 1$. Since $(y + 1)e^{-ky} \leq C_{y_{\star}} e^{-ky_{\star}/2}$ for $y \geq y_{\star}$ and $k \geq 1$, then the right-hand side is bounded by
\[
C_{y_{\star}}'(k + 1)^b\left(\rho^{-1} e^{-y_{\star}/2}\right)^k,
\]
which is summable in $k$. The differentiated series therefore converges uniformly for $\Delta \in \mathcal{C}$ and $y \geq y_{\star}$. Consequently, every fixed mixed derivative of $q$ is obtained by differentiating \eqref{eq:q.preliminary.mode.expansion} term by term.

It remains to prove the asserted polynomial structure of the coefficients. The hypergeometric differential equation \cite[Eq.~15.10.1]{Daalhuis2010} states that $\mathcal{F}(\Delta,z) = {}_2F_1(\Delta,\Delta;2\Delta;z)$ satisfies
\begin{equation}\label{eq:diff.eq.mathcal.F}
z(1 - z)\partial_z^2\mathcal{F}(\Delta,z) + [2\Delta - (2\Delta + 1)z]\partial_z\mathcal{F}(\Delta,z) - \Delta^2\mathcal{F}(\Delta,z) = 0.
\end{equation}
Since $G_{\Delta}(z) = z^{\Delta}\mathcal{F}(\Delta,z)$, then by substituting $\mathcal{F}(\Delta,z) = z^{-\Delta} G_{\Delta}(z)$ in \eqref{eq:diff.eq.mathcal.F}, we obtain
\begin{equation}\label{eq:G.differential.equation}
z^2\{(1 - z)\partial_z^2G_{\Delta}(z) - \partial_zG_{\Delta}(z)\} = \Delta(\Delta - 1)G_{\Delta}(z).
\end{equation}
For $0 < t < 1$, define
\[
\widehat{q}(\Delta,t) \leqdef B(\Delta,\Delta)G_{\Delta}(1 - t).
\]
Then $q(\Delta,y) = \widehat{q}(\Delta,e^{-y})$. With $t = e^{-y}$, we apply the chain rule to obtain
\begin{equation}\label{eq:chain.rule.q}
\begin{aligned}
\partial_y q(\Delta,y)
&= -t\partial_t\widehat{q}(\Delta,t), \\
\partial_y^2 q(\Delta,y)
&= t\partial_t\widehat{q}(\Delta,t) + t^2\partial_t^2\widehat{q}(\Delta,t) = t\{t\partial_t^2\widehat{q}(\Delta,t) + \partial_t\widehat{q}(\Delta,t)\}.
\end{aligned}
\end{equation}
Since $z = 1 - t$, then
\begin{align*}
\partial_t\widehat{q}(\Delta,t) = -B(\Delta,\Delta)\partial_zG_{\Delta}(z), \quad
\partial_t^2\widehat{q}(\Delta,t) = B(\Delta,\Delta)\partial_z^2G_{\Delta}(z),
\end{align*}
and now multiplying \eqref{eq:G.differential.equation} by $B(\Delta,\Delta)$ leads to the identity
\begin{equation}\label{eq:diff.eq.hat.q}
(1 - t)^2\{t\partial_t^2\widehat{q}(\Delta,t) + \partial_t\widehat{q}(\Delta,t)\} = \lambda\widehat{q}(\Delta,t),
\end{equation}
where we recall that $\lambda = \Delta (\Delta-1)$. Combining \eqref{eq:chain.rule.q} and \eqref{eq:diff.eq.hat.q}, and using $t = e^{-y}$, gives
\begin{equation}\label{eq:q.y.diff.eq}
\partial_y^2 q(\Delta,y)
= \lambda\frac{e^{-y}}{(1 - e^{-y})^2}q(\Delta,y)
= \lambda\sum_{r=1}^{\infty}r e^{-ry}q(\Delta,y).
\end{equation}
For each $k \in \N_0$, the second derivative of the $k$th term in the series in \eqref{eq:q.preliminary.mode.expansion} is
\[
\begin{aligned}
\partial_y^2\big[e^{-ky}\{\widetilde{P}_k(\Delta)y + \widetilde{Q}_k(\Delta)\}\big]
&= k^2 e^{-ky}\{\widetilde{P}_k(\Delta)y + \widetilde{Q}_k(\Delta)\} - k e^{-ky}\widetilde{P}_k(\Delta) - k e^{-ky}\widetilde{P}_k(\Delta) \\
&= e^{-ky}\big[k^2\widetilde{P}_k(\Delta)y + k^2\widetilde{Q}_k(\Delta) - 2k\widetilde{P}_k(\Delta)\big].
\end{aligned}
\]
Substituting the series expansion from \eqref{eq:q.preliminary.mode.expansion} into the right-hand side of \eqref{eq:q.y.diff.eq} yields a product of two sums. When the factor $r e^{-ry}$ is multiplied by a term $e^{-jy}\{\widetilde{P}_j(\Delta)y + \widetilde{Q}_j(\Delta)\}$ from the expansion, the resulting exponent is $e^{-(j+r)y}$. Therefore, this product contributes to the overall coefficient of $e^{-ky}$ precisely when $j + r = k$ (or equivalently, $r = k - j$). Since the sum over $r$ starts at $r = 1$, then the index $j$ ranges from $0$ to $k - 1$. Equating the coefficients of $y e^{-ky}$ and $e^{-ky}$ on both sides of \eqref{eq:q.y.diff.eq} for $k \geq 1$, we obtain the recurrence relations
\begin{equation}\label{eq:P.Q.recurrences}
k^2\widetilde{P}_k = \lambda\sum_{j=0}^{k - 1}(k - j)\widetilde{P}_j, \qquad
k^2\widetilde{Q}_k - 2k\widetilde{P}_k = \lambda\sum_{j=0}^{k - 1}(k - j)\widetilde{Q}_j.
\end{equation}
Since $\widetilde{P}_0 = 1$, the first recurrence in \eqref{eq:P.Q.recurrences} shows by induction that there is a polynomial $P_k$ in one variable such that $\widetilde{P}_k(\Delta) = P_k(\lambda)$ and $\deg P_k = k$. If $P_{k - 1}$ has leading coefficient $((k - 1)!)^{-2}$, then the term with $j = k - 1$ is the only term on the right-hand side that contributes to degree $k$. The leading coefficient of $P_k$ is therefore
\[
\frac{1}{k^2((k - 1)!)^2} = \frac{1}{(k!)^2}.
\]
Thus, $P_k$ has the asserted leading coefficient for every $k \in \N_0$.

It remains to treat $\widetilde{Q}_k$. Let $S_0 = 0$, so that $\widetilde{Q}_0(\Delta) = P_0(\lambda)h(\Delta) + S_0(\lambda)$. Suppose by induction that
\[
\widetilde{Q}_j(\Delta) = P_j(\lambda)h(\Delta) + S_j(\lambda), \qquad 0 \leq j \leq k - 1,
\]
where each $S_j$ is a polynomial in $\lambda$ of degree at most $j$. Substitution in the second recurrence of \eqref{eq:P.Q.recurrences}, followed by the first recurrence, gives
\[
\widetilde{Q}_k(\Delta)
= \frac{2}{k}P_k(\lambda) + \frac{\lambda}{k^2}\sum_{j=0}^{k - 1}(k - j)\big[P_j(\lambda)h(\Delta) + S_j(\lambda)\big]
= P_k(\lambda)h(\Delta) + S_k(\lambda),
\]
where
\[
S_k(\lambda) \leqdef \frac{2}{k}P_k(\lambda) + \frac{\lambda}{k^2}\sum_{j=0}^{k - 1}(k - j)S_j(\lambda).
\]
This formula shows that $S_k$ is a polynomial in $\lambda$ of degree at most $k$. The assertion therefore follows by induction. Defining $Q_k(\Delta) \leqdef \widetilde{Q}_k(\Delta)$ in \eqref{eq:q.preliminary.mode.expansion}, we obtain \eqref{eq:q.mode.expansion} and all the assertions of the lemma. This concludes the proof of Lemma~\ref{lem:q.boundary.expansion}.
$\qed$

\subsection{The two Wronskians in the large-\texorpdfstring{$y$}{y} asymptotic formula for \texorpdfstring{$D_{2m+1}$}{D\_\{2m+1\}} in Lemma~\ref{lem:confluent.boundary.asymptotic}}

To evaluate the determinant of mixed derivatives $D_{2m+1}(\Delta, y) = \det[\partial_{\Delta}^i\partial_y^j q(\Delta,y)]_{i,j=0}^{2m}$, the Cauchy--Binet formula requires summing over all possible selections of $2m+1$ terms from the infinite series in \eqref{eq:q.mode.expansion}. As $y \to \infty$, the dominant contribution comes from the selection with the slowest exponential decay. The specific selection of $2m+1$ terms that can produce the leading asymptotic term $y e^{-m^2y}$ chooses the two products $P_k(\lambda)ye^{-ky}$ and $Q_k(\Delta)e^{-ky}$ for every $k = 0,\ldots,m - 1$, along with the single additional product $P_m(\lambda)ye^{-my}$. The $y$-dependent factors of this selection, in their Cauchy--Binet order, are
\begin{equation}\label{eq:y.dependent.factors}
y,1,ye^{-y},e^{-y},\ldots,ye^{-(m - 1)y},e^{-(m - 1)y},ye^{-my}.
\end{equation}
The next lemma evaluates the Wronskian of exactly this list and determines its exponential factor, its sign, and its coefficient of $y$.

\begin{lemma}[Wronskian of the selected $y$-dependent factors in \eqref{eq:y.dependent.factors}]\label{lem:mode.Wronskian}
Let $m \in \N$ and define
\begin{equation}\label{eq:b.m}
b_m \leqdef \prod_{0\leq r<s\leq m - 1}(s - r)^4\prod_{r=0}^{m - 1}(m - r)^2 > 0.
\end{equation}
Then, as $y \to \infty$,
\begin{equation}\label{eq:mode.Wronskian}
W_y(y,1,ye^{-y},e^{-y},\ldots,ye^{-(m - 1)y},e^{-(m - 1)y},ye^{-my}) = (-1)^m b_m y e^{-m^2y} + O(e^{-m^2y}).
\end{equation}
\end{lemma}

This lemma is a direct specialization of the standard confluent Vandermonde determinant formula; see \citet[p.~117, Eqs.~(1.1)--(1.3)]{Gautschi1962} for the confluent construction (the limiting procedure that generates the $y e^{-ky}$ terms \textit{via} differentiation with respect to the exponent parameters), and \citet[Eq.~(2.4), p.~139]{BatenkovYomdin2013} for the determinant formula. The proof is included for completeness and to keep track of the sign and the coefficient of $y$ in \eqref{eq:mode.Wronskian}.

\begin{proof}[Proof of Lemma~\ref{lem:mode.Wronskian}]
Define $E_{\alpha}(y) \leqdef e^{\alpha y}$. For $L \in \N_0$ and distinct $\beta_0,\ldots,\beta_L$, since
\[
\frac{\rd^j}{\rd y^j}E_{\alpha}(y) = \alpha^j e^{\alpha y},
\]
the ordinary Vandermonde determinant gives
\[
W_y(E_{\beta_0},\ldots,E_{\beta_L}) = e^{(\beta_0 + \cdots + \beta_L)y}\prod_{0\leq i<j\leq L}(\beta_j - \beta_i).
\]
Fix distinct real numbers $\alpha_0,\ldots,\alpha_m$, and take the nonzero $\varepsilon_r$ sufficiently small that all $2m + 1$ exponents $\alpha_0,\alpha_0 + \varepsilon_0,\ldots,\alpha_{m - 1},\alpha_{m - 1} + \varepsilon_{m - 1},\alpha_m$ remain distinct. For each $r = 0,\ldots,m - 1$, pair $\alpha_r$ with $\alpha_r + \varepsilon_r$, while $\alpha_m$ remains unpaired. Subtract the row belonging to $E_{\alpha_r}$ from the row belonging to $E_{\alpha_r + \varepsilon_r}$, divide that row by $\varepsilon_r$, and let $\varepsilon_r \to 0$. The new row converges to the row belonging to
\[
\partial_{\alpha_r}E_{\alpha_r}(y) = yE_{\alpha_r}(y).
\]
In the Vandermonde product, the difference within the $r$th pair is $(\alpha_r + \varepsilon_r) - \alpha_r = \varepsilon_r$, which is canceled by the division of the corresponding row by $\varepsilon_r$. For $r < s$, the four differences between the two exponents in the $r$th pair and the two exponents in the $s$th pair all converge to $\alpha_s - \alpha_r$. The two differences between the exponents in the $r$th pair and the unpaired exponent $\alpha_m$ both converge to $\alpha_m - \alpha_r$. Therefore,
\begin{multline}
W_y(E_{\alpha_0},\partial_{\alpha_0}E_{\alpha_0},\ldots,E_{\alpha_{m - 1}},\partial_{\alpha_{m - 1}}E_{\alpha_{m - 1}},E_{\alpha_m}) \\
= e^{(2\sum_{r=0}^{m - 1}\alpha_r + \alpha_m)y}\prod_{0\leq r<s\leq m - 1}(\alpha_s - \alpha_r)^4\prod_{r=0}^{m - 1}(\alpha_m - \alpha_r)^2. \label{eq:confluent.Vandermonde}
\end{multline}

Substitute $\alpha_r = -r$, $r = 0,\ldots,m$, into \eqref{eq:confluent.Vandermonde}. The first $m$ pairs in the target Wronskian \eqref{eq:mode.Wronskian} are ordered as $(\partial_{\alpha_r}E_{\alpha_r},E_{\alpha_r})$, whereas the corresponding pairs in \eqref{eq:confluent.Vandermonde} are ordered as $(E_{\alpha_r},\partial_{\alpha_r}E_{\alpha_r})$. Reversing these $m$ pairs contributes $(-1)^m$. Also,
\[
2\sum_{r=0}^{m - 1}(-r) - m = -m^2.
\]
To identify the dependence on $y$, note that
\[
\frac{\rd^j}{\rd y^j}(y e^{\alpha y}) = e^{\alpha y}\{y\alpha^j + j\alpha^{j - 1}\}.
\]
Here the second summand is absent when $j = 0$. Consider the evaluated Wronskian matrix in \eqref{eq:mode.Wronskian}. For every $r = 0,\ldots,m - 1$, subtract $y$ times the row associated with $E_{-r}$ from the row associated with $yE_{-r}$. In derivative column $0$, the new entry is zero. For $j \geq 1$, the entry in derivative column $j$ is
\[
j(-r)^{j - 1}e^{-ry}.
\]
Thus, these $m$ rows no longer contain a polynomial factor in $y$. Factor $e^{-ry}$ from each of the two rows belonging to degree $r < m$ and factor $e^{-my}$ from the final row belonging to $ye^{-my}$. The product of the extracted factors is
\[
e^{-\left(2\sum_{r=0}^{m - 1}r + m\right)y} = e^{-m^2y}.
\]
After these factors are removed, only the final row depends on $y$, and it depends affinely on $y$. The coefficient of $y$ in that row is the row associated with $E_{-m}$. Hence, after the first $m$ pairs have been reversed, \eqref{eq:confluent.Vandermonde} shows that the coefficient of $y$ in the Wronskian \eqref{eq:mode.Wronskian} is $(-1)^m b_m e^{-m^2y}$. Its remaining part is a constant multiple of $e^{-m^2y}$, which proves \eqref{eq:mode.Wronskian}.
\end{proof}

It is a consequence of the Cauchy--Binet formula that the determinant of mixed derivatives $D_{2m+1}(\Delta, y)$ is a sum of products, where each product consists of a $y$-Wronskian multiplied by a corresponding $\Delta$-Wronskian. Lemma~\ref{lem:mode.Wronskian} evaluated the $y$-Wronskian for the dominant selection. To complete the asymptotic evaluation in Lemma~\ref{lem:confluent.boundary.asymptotic}, this result must be multiplied by the $\Delta$-Wronskian of the same selection.

Recall that the winning selection consists of the terms $P_k(\lambda)ye^{-ky}$ and $Q_k(\Delta)e^{-ky}$ for $k = 0,\ldots,m - 1$, along with the single term $P_m(\lambda)ye^{-my}$. While Lemma~\ref{lem:mode.Wronskian} isolated their $y$-dependent halves, their corresponding $\Delta$-dependent halves are precisely $P_k(\lambda)$ and $Q_k(\Delta)$. Consequently, the required $\Delta$-Wronskian for this selection is $W_{\Delta}(P_0,Q_0,\ldots,P_{m - 1},Q_{m - 1},P_m)$. To evaluate this $\Delta$-half of the product, the proof of Lemma~\ref{lem:confluent.boundary.asymptotic} applies a change of variables to $\xi = (\tfrac{1}{2} - \Delta)^2$. This converts the $\Delta$-derivatives into $\xi$-derivatives, yielding the relation
\[
W_{\Delta}(P_0,Q_0,\ldots,P_{m - 1},Q_{m - 1},P_m) = \kappa_m(2\Delta - 1)^{m(2m + 1)}W_{\xi}(1,c,\xi,\xi c,\ldots,\xi^{m - 1},\xi^{m - 1}c,\xi^m),
\]
where $W_{\xi}$ denotes a Wronskian with respect to $\xi$, $c(\xi) = -2\gamma - 2\psi(\tfrac{1}{2} - \xi^{1/2})$ as defined in \eqref{eq:xi.c.definition}, and
\begin{equation}\label{eq:kappa.m}
\kappa_m \leqdef \frac{1}{(m!)^2}\prod_{k=0}^{m - 1}\frac{1}{(k!)^4} > 0.
\end{equation}
The next identity explicitly resolves this new $W_{\xi}$ Wronskian. The key observation is that, after the evaluated matrix is transposed, taking successive $\xi$-derivatives of the polynomial terms $\xi^k$ leaves them with only one nonzero entry per column, allowing the determinant to be expanded easily.

\begin{lemma}[Interlaced Wronskian identity]\label{lem:interlaced.Wronskian}
Let $I \subseteq \R$ be an open interval. For $m \in \N$, let $c:I \to \R$ be $2m$ times continuously differentiable. Then, for $x \in I$,
\begin{equation}\label{eq:interlaced.Wronskian}
W_x(1,c,x,xc,\ldots,x^{m - 1},x^{m - 1}c,x^m) = (-1)^m\left(\prod_{r=0}^{2m}r!\right)\det\!\left[\frac{c^{(i + j)}(x)}{(i + j)!}\right]_{i,j=1}^m.
\end{equation}
\end{lemma}

\begin{proof}[Proof of Lemma~\ref{lem:interlaced.Wronskian}]
Fix $x_0 \in I$, let $t = x - x_0$, and define $\widetilde{c}(t) \leqdef c(x_0 + t)$. Translation of the independent variable does not change the derivatives, so the Wronskian at $x = x_0$ equals the corresponding Wronskian at $t = 0$. Within the polynomial part of the interlaced list, the change-of-basis matrix from $1,x,\ldots,x^m$ to $1,t,\ldots,t^m$ is triangular with ones on the main diagonal. Consequently, this substitution preserves the value of the determinant. The same is true of the replacement of $c(x),xc(x),\ldots,x^{m - 1}c(x)$ by $\widetilde{c}(t),t\widetilde{c}(t),\ldots,t^{m - 1}\widetilde{c}(t)$. Hence, the entire ordered list is transformed by a matrix with determinant one, so its Wronskian is unchanged. It is therefore enough to prove the identity at $t = 0$. For notational simplicity, rename $t$ as $x$ and $\widetilde{c}$ as $c$.

Transpose the Wronskian matrix and divide the row corresponding to derivative order $r$ by $r!$, for $r = 0,\ldots,2m$. The column belonging to $x^k$ then has its only nonzero entry, equal to one, in row $k$. For $r \geq k$, the entry in row $r$ of the column belonging to $x^k c$ is
\[
\frac{1}{r!}\frac{\rd^r}{\rd x^r}\{x^k c(x)\}\bigg|_{x=0} = \frac{c^{(r - k)}(0)}{(r - k)!},
\]
and it is zero for $r < k$. This arrangement of the interlaced basis functions $1, c, x, xc, \ldots, x^m$ as columns and the derivative orders $r$ as rows makes explicit the sparse structure of the matrix:
\[
\begin{pmatrix}
1      & c(0)                      & 0      & 0                           & \cdots & 0 \\[0.5ex]
0      & c'(0)                     & 1      & c(0)                        & \cdots & 0 \\[1ex]
0      & \dfrac{c''(0)}{2!}         & 0      & c'(0)                       & \cdots & 0 \\
\vdots & \vdots                    & \vdots & \vdots                      & \ddots & \vdots \\
0      & \dfrac{c^{(m)}(0)}{m!}     & 0      & \dfrac{c^{(m-1)}(0)}{(m-1)!} & \cdots & 1 \\
\vdots & \vdots                    & \vdots & \vdots                      & \ddots & \vdots \\
0      & \dfrac{c^{(2m)}(0)}{(2m)!} & 0      & \dfrac{c^{(2m-1)}(0)}{(2m-1)!} & \cdots & 0
\end{pmatrix}.
\]
Since every column associated with a pure polynomial $x^k$ contains exactly one nonzero entry (the $1$ on the staggered diagonal), a Laplace expansion along these columns trivially reduces the determinant to the remaining $c$-dependent entries.

We now expand the normalized determinant along the polynomial columns $1,x,\ldots,x^m$. These columns occupy the zero-based positions $0,2,\ldots,2m$ and select the rows $0,1,\ldots,m$. The sign of this expansion is
\[
(-1)^{\{0 + \cdots + m\} + \{0 + 2 + \cdots + 2m\}} = (-1)^{m(m + 1)/2 + m(m + 1)}.
\]
The rows that remain have orders $m + 1,\ldots,2m$, and the remaining columns are $c,xc,\ldots,x^{m - 1}c$. Reverse these $m$ columns, which contributes $(-1)^{m(m - 1)/2}$. If the remaining row is indexed as $m + i$ and the reversed column as $x^{m - j}c$, where $1 \leq i,j \leq m$, its entry is
\[
\frac{c^{(m + i - (m - j))}(0)}{(m + i - (m - j))!} = \frac{c^{(i + j)}(0)}{(i + j)!}.
\]
The combined sign is
\[
(-1)^{m(m + 1)/2 + m(m + 1) + m(m - 1)/2} = (-1)^{m(2m + 1)} = (-1)^m.
\]
Finally, restoring the row factors that were divided out multiplies the determinant by $\prod_{r=0}^{2m}r!$. Translating back proves \eqref{eq:interlaced.Wronskian} at the original point $x_0$. Since $x_0$ was arbitrary, the result of the lemma follows.
\end{proof}

\subsection{Proof of Lemma~\ref{lem:confluent.boundary.asymptotic}}

Let $N \leqdef 2m + 1$. Write \eqref{eq:q.mode.expansion} as a sum of separated terms by defining, for $k \in \N_0$,
\[
\begin{aligned}
A_{2k}(\Delta) &\leqdef P_k(\lambda), & A_{2k + 1}(\Delta) &\leqdef Q_k(\Delta), \\
B_{2k}(y) &\leqdef ye^{-ky}, & B_{2k + 1}(y) &\leqdef e^{-ky}.
\end{aligned}
\]
For a term in the series indexed by $r$, let
\[
d(r) \leqdef \left\lfloor\frac{r}{2}\right\rfloor
\]
denote its original summation index $k$. For each $r \in \N_0$, we call the function $(\Delta,y) \mapsto A_r(\Delta)B_r(y)$ \textit{mode} $r$, and we call $d(r)$ its \textit{exponential degree}. Thus, each index $r$ identifies a single mode; the repetition occurs only among the exponential degrees, since the two distinct modes $2k$ and $2k + 1$ both have exponential degree $k$.

The proof has three steps. First, by applying the discrete Cauchy--Binet formula, we obtain the determinant of mixed derivatives of a truncated kernel (see \eqref{eq:truncated.kernel}) in the form of a sum over sets of $N = 2m + 1$ distinct mode indices. Second, we use the sum of the selected exponential degrees to identify the slowest exponential decay. The minimum total degree is $m^2$ and is attained by exactly two selections, whose index sets are
\[
\{0,1,\ldots,2m - 1,2m\} \quad\text{and}\quad \{0,1,\ldots,2m - 1,2m + 1\}.
\]
Both selections contain the two modes of each degree $0,\ldots,m - 1$ and exactly one of the two modes of degree $m$. The first contains mode $2m$, whose $y$-dependent factor is $B_{2m}(y) = ye^{-my}$, whereas the second contains mode $2m + 1$, whose $y$-dependent factor is $B_{2m + 1}(y) = e^{-my}$. As shown by the row operations below, the paired factors of each lower degree cannot contribute an uncanceled factor $y$. Consequently, only the first selection can contribute a term proportional to $y e^{-m^2y}$. Third, for this first selection, Lemma~\ref{lem:mode.Wronskian} evaluates the $y$-Wronskian
\[
W_y(B_0,\ldots,B_{2m}) = W_y(y,1,ye^{-y},e^{-y},\ldots,ye^{-(m - 1)y},e^{-(m - 1)y},ye^{-my}),
\]
while the triangular calculation below and Lemma~\ref{lem:interlaced.Wronskian} evaluate the corresponding $\Delta$-Wronskian
\[
W_{\Delta}(A_0,\ldots,A_{2m}) = W_{\Delta}(P_0,Q_0,\ldots,P_{m - 1},Q_{m - 1},P_m).
\]

Fix $K \geq m$ and define the truncated kernel
\begin{equation}\label{eq:truncated.kernel}
q_K(\Delta,y) \leqdef \sum_{r=0}^{2K + 1}A_r(\Delta)B_r(y).
\end{equation}
Its $N \times N$ matrix of mixed derivatives factors as
\[
\left[\partial_{\Delta}^i\partial_y^j q_K(\Delta,y)\right]_{i,j=0}^{N - 1}
= \left[\partial_{\Delta}^i A_r(\Delta)\right]_{\substack{0\leq i\leq N - 1\\0\leq r\leq 2K + 1}}
\left[\partial_y^j B_r(y)\right]_{\substack{0\leq r\leq 2K + 1\\0\leq j\leq N - 1}}.
\]
For selected indices $0 \leq r_1 < \cdots < r_N \leq 2K + 1$, the corresponding minor of the first factor is the transpose of the Wronskian matrix $[\partial_{\Delta}^i A_{r_{\ell}}(\Delta)]_{1\leq \ell\leq N, 0\leq i\leq N - 1}$, so its determinant is $W_{\Delta}(A_{r_1},\ldots,A_{r_N})$. The corresponding minor of the second factor is the Wronskian matrix of $B_{r_1},\ldots,B_{r_N}$, so its determinant is $W_y(B_{r_1},\ldots,B_{r_N})$. The finite Cauchy--Binet formula therefore gives
\begin{equation}\label{eq:finite.CB.boundary}
\det\!\left[\partial_{\Delta}^i\partial_y^j q_K(\Delta,y)\right]_{i,j=0}^{N - 1}
= \sum_{0\leq r_1<\cdots<r_N\leq 2K + 1}W_{\Delta}(A_{r_1},\ldots,A_{r_N})W_y(B_{r_1},\ldots,B_{r_N}).
\end{equation}

We next justify passage to the infinite series. Fix a compact interval $J \subseteq (0,1/2)$ of $\Delta$-values. By Lemma~\ref{lem:q.boundary.expansion}, the two generating series $\sum_{k=0}^{\infty}P_k(\lambda)t^k$ and $\sum_{k=0}^{\infty}Q_k(\Delta)t^k$, together with their $\Delta$-derivatives of orders $0,\ldots,N - 1$, converge uniformly for $\Delta \in J$ and $t$ in compact subsets of the unit disk $|t| < 1$.

Recall our earlier definition that $A_{2k}(\Delta) = P_k(\lambda)$ and $A_{2k+1}(\Delta) = Q_k(\Delta)$. To bound the growth of these terms, we rely on Cauchy's coefficient estimate, which bounds the $k$-th coefficient of an analytic function by its maximum modulus on a circle of radius $R$, divided by $R^k$. For any fixed radius $R \in (0,1)$, the uniform convergence on the compact domain $J \times \{|t| = R\}$ guaranteed by Lemma~\ref{lem:q.boundary.expansion} implies that these generating series and their derivatives are bounded by a universal constant $C$. Applying Cauchy's estimate to extract the coefficients of the generating functions yields
\[
\left|\partial_{\Delta}^i A_{2k}(\Delta)\right| + \left|\partial_{\Delta}^i A_{2k + 1}(\Delta)\right| \leq C R^{-k}, \qquad 0 \leq i \leq N - 1.
\]
For the $y$-dependent functions, direct differentiation gives
\[
\left|\partial_y^j B_{2k}(y)\right| + \left|\partial_y^j B_{2k + 1}(y)\right| \leq C(k + 1)^N(y + 1)e^{-ky}, \qquad 0 \leq j \leq N - 1.
\]
After both Wronskians in each summand of \eqref{eq:finite.CB.boundary} are expanded as sums over permutations, the absolute value of the product corresponding to selected mode indices $r_1,\ldots,r_N$, with degrees $k_j = d(r_j)$, is bounded by a constant times
\[
(y + 1)^N \prod_{j=1}^N(k_j + 1)^N \left(\frac{e^{-y}}{R}\right)^{k_1 + \cdots + k_N}.
\]
For fixed $N$, the number of selections with $k_1 + \cdots + k_N = \ell$ is bounded by $(\ell + 1)^{C_N}$ for some constant $C_N$ depending only on $N$. Thus, when $y$ is large enough that $e^{-y} < R$, the Cauchy--Binet sum converges absolutely and uniformly for $\Delta \in J$. The mixed derivatives of $q_K$ also converge to the corresponding mixed derivatives of $q$. Letting $K \to \infty$ in \eqref{eq:finite.CB.boundary} therefore gives the absolutely convergent expansion
\[
D_N(\Delta,y) = \sum_{0\leq r_1<\cdots<r_N}W_{\Delta}(A_{r_1},\ldots,A_{r_N})W_y(B_{r_1},\ldots,B_{r_N}).
\]

Since there are two modes of each exponential degree $k$, the smallest possible sum of the degrees of $2m + 1$ distinct modes is
\[
2\sum_{k=0}^{m - 1}k + m = m^2.
\]
Exactly two selections attain this value. Selecting the indices $0,\ldots,2m$ gives the corresponding $\Delta$-list and $y$-list:
\[
\begin{aligned}
\Delta\text{-list:}& \quad (P_0,Q_0,\ldots,P_{m - 1},Q_{m - 1},P_m), \\
y\text{-list:}& \quad (y,1,\ldots,ye^{-(m - 1)y},e^{-(m - 1)y},ye^{-my}).
\end{aligned}
\]
The corresponding Cauchy--Binet summand is the product of the Wronskians of these two lists. Lemma~\ref{lem:mode.Wronskian} shows that its $y$-Wronskian is $(-1)^m b_m y e^{-m^2y} + O(e^{-m^2y})$. The other minimizing selection consists of the indices $0,\ldots,2m - 1,2m + 1$; it replaces $P_m$ by $Q_m$ in the $\Delta$-list and $ye^{-my}$ by $e^{-my}$ in the $y$-list. In the $y$-Wronskian for this second selection, for every $k < m$, subtract $y$ times the row associated with $e^{-ky}$ from the row associated with $ye^{-ky}$. At a fixed value of $y$, this is an ordinary row operation on the evaluated Wronskian matrix. It removes every polynomial factor in $y$. Factoring $e^{-ky}$ from each of the two rows of degree $k < m$ and $e^{-my}$ from the last row then gives the total factor
\[
e^{-\left(2\sum_{k=0}^{m - 1}k + m\right)y} = e^{-m^2y}.
\]
After the row operations and the extraction of the exponential factors, the $y$-Wronskian is a constant multiple of $e^{-m^2y}$. The $\Delta$-Wronskian is independent of $y$ and is bounded when $\Delta$ ranges over a compact subinterval of $(0,1/2)$. Hence, the complete Cauchy--Binet summand for the second minimizing selection is $e^{-m^2y}O(1)$. The first minimizing selection is therefore the only one that can contribute a term proportional to $y e^{-m^2y}$.

The estimates above show that, uniformly for $\Delta$ in the compact set $J$, the sum of all these remaining contributions is bounded by
\begin{equation}\label{eq:tail.bound}
C(y + 1)^N\sum_{\ell\geq m^2 + 1}(\ell + 1)^{C_N}\left(\frac{e^{-y}}{R}\right)^{\ell} = O((y + 1)^Ne^{-(m^2 + 1)y}) = e^{-m^2y}O(1),
\end{equation}
where the implicit constants in the $O$-notation hold as $y \to \infty$ and are uniform with respect to $\Delta \in J$.

It remains to calculate the $\Delta$-Wronskian in the first minimizing selection. Recall from \eqref{eq:h.definition} that $\lambda = \Delta(\Delta - 1)$ and $h(\Delta) = -2\gamma - 2\psi(\Delta)$. Let $a_k \leqdef (k!)^{-2}$ and consider the ordered list
\[
\mathcal{U}_m \leqdef (1,h,\lambda,\lambda h,\ldots,\lambda^{m - 1},\lambda^{m - 1}h,\lambda^m).
\]
For $0 \leq k \leq m$, the polynomial $P_k$ has degree $k$ with leading coefficient $a_k$. It expands as $P_k(\lambda) = a_k\lambda^k + \smash{\sum_{j=0}^{k - 1}p_{k,j}\lambda^j}$ for suitable real constants $p_{k,j}$. For $0 \leq k < m$, the decomposition $Q_k = P_k h + S_k$ with $\deg S_k \leq k$ from Lemma~\ref{lem:q.boundary.expansion} guarantees that $Q_k$ consists of the leading term $a_k\lambda^kh$ plus a linear combination of strictly lower-order functions.

Consequently, expressing the ordered list $(P_0,Q_0,\ldots,P_{m - 1},Q_{m - 1},P_m)^{\top}$ in terms of the interlaced basis $\mathcal{U}_m = (1,h,\lambda,\lambda h,\ldots,\lambda^m)^{\top}$ results in a change-of-basis matrix that is lower-triangular. With asterisks $(*)$ denoting the remaining constants from the lower-order terms, this linear transformation takes the explicit form:
\[
\begin{pmatrix}
P_0 \\ Q_0 \\ P_1 \\ Q_1 \\ \vdots \\ P_{m - 1} \\ Q_{m - 1} \\ P_m
\end{pmatrix}
=
\begin{pmatrix}
a_0 & 0   & 0   & 0   & \cdots & 0 & 0 & 0 \\
*   & a_0 & 0   & 0   & \cdots & 0 & 0 & 0 \\
*   & 0   & a_1 & 0   & \cdots & 0 & 0 & 0 \\
*   & *   & *   & a_1 & \cdots & 0 & 0 & 0 \\
\vdots & \vdots & \vdots & \vdots & \ddots & \vdots & \vdots & \vdots \\
*   & 0   & *   & 0   & \cdots & a_{m - 1} & 0 & 0 \\
*   & *   & *   & *   & \cdots & * & a_{m - 1} & 0 \\
*   & 0   & *   & 0   & \cdots & * & 0 & a_m
\end{pmatrix}
\begin{pmatrix}
1 \\ h \\ \lambda \\ \lambda h \\ \vdots \\ \lambda^{m - 1} \\ \lambda^{m - 1}h \\ \lambda^m
\end{pmatrix}.
\]
Since the determinant of a triangular matrix is the product of its diagonal entries, this change of basis scales the corresponding Wronskian exactly by the factor $a_0^2 a_1^2 \cdots a_{m - 1}^2 a_m$.
The determinant of this lower-triangular change of basis is $a_m\prod_{k=0}^{m - 1}a_k^2 = \kappa_m$, where $\kappa_m$ was defined in \eqref{eq:kappa.m}. Therefore,
\begin{equation}\label{eq:lower.triangular.reduction}
W_{\Delta}(P_0,Q_0,P_1,Q_1,\ldots,P_{m - 1},Q_{m - 1},P_m) = \kappa_m W_{\Delta}(1,h,\lambda,\lambda h,\ldots,\lambda^{m - 1},\lambda^{m - 1}h,\lambda^m).
\end{equation}

On the interval $0 < \Delta < 1/2$, equations~\eqref{eq:h.definition} and~\eqref{eq:xi.c.definition} yield
\[
h(\Delta) = c(\xi), \qquad \lambda = \Delta(\Delta - 1) = \xi - \frac{1}{4}.
\]
Consequently, the ordered list of $\Delta$-dependent functions $(1,h,\lambda,\lambda h,\ldots,\lambda^{m - 1},\lambda^{m - 1}h,\lambda^m)$ is obtained from the corresponding $\xi$-dependent list $(1,c,\xi,\xi c,\ldots,\xi^{m - 1},\xi^{m - 1}c,\xi^m)$ by a constant triangular change of basis with diagonal entries equal to one. More generally, the change-of-variable formula for Wronskians proved by \citet[Theorem~27, pp.~98--99]{Kiselev2004} gives
\[
W_{\Delta}(f_1\circ\xi,\ldots,f_N\circ\xi) = (\xi'(\Delta))^{N(N - 1)/2}W_{\xi}(f_1,\ldots,f_N).
\]
Since $N = 2m + 1$ and $\xi'(\Delta) = 2\Delta - 1$, it follows that
\begin{equation}\label{eq:change.of.variables}
\begin{aligned}
&W_{\Delta}(1,h,\lambda,\lambda h,\ldots,\lambda^{m - 1},\lambda^{m - 1}h,\lambda^m) \\
&\hspace{30mm}= (2\Delta - 1)^{m(2m + 1)}W_{\xi}(1,c,\xi,\xi c,\ldots,\xi^{m - 1},\xi^{m - 1}c,\xi^m).
\end{aligned}
\end{equation}
Combining the lower-triangular reduction \eqref{eq:lower.triangular.reduction}, the change of variables \eqref{eq:change.of.variables}, and Lemma~\ref{lem:interlaced.Wronskian} gives the $\Delta$-Wronskian in the first minimizing selection as
\begin{equation}\label{eq:W.Delta}
W_{\Delta}(P_0,Q_0,\ldots,P_{m - 1},Q_{m - 1},P_m) = \kappa_m(2\Delta - 1)^{m(2m + 1)}(-1)^m\left(\prod_{r=0}^{2m}r!\right)\det H_m(\xi).
\end{equation}
For the same selection, Lemma~\ref{lem:mode.Wronskian} gives its $y$-Wronskian as
\begin{equation}\label{eq:W.y}
W_y(y,1,\ldots,ye^{-(m - 1)y},e^{-(m - 1)y},ye^{-my}) = (-1)^m b_m y e^{-m^2y} + O(e^{-m^2y}).
\end{equation}
The factor $(-1)^m$ in \eqref{eq:W.Delta} cancels the corresponding factor in the leading term of \eqref{eq:W.y}. Hence, the product belonging to the first minimizing selection is
\begin{equation}\label{eq:first.min.contrib}
C_m e^{-m^2y}\left\{(2\Delta - 1)^{m(2m + 1)}y\det H_m(\xi) + O(1)\right\},
\end{equation}
where $b_m$ is defined in \eqref{eq:b.m}, $\kappa_m$ is defined in \eqref{eq:kappa.m}, and
\[
C_m
\leqdef \kappa_m b_m\prod_{r=0}^{2m}r!
= \left\{\frac{1}{(m!)^2}\prod_{k=0}^{m - 1}\frac{1}{(k!)^4}\right\}\left\{\prod_{0\leq r<s\leq m - 1}(s - r)^4\prod_{r=0}^{m - 1}(m - r)^2\right\}\prod_{r=0}^{2m}r! > 0.
\]
For the compact interval $J \subseteq (0,1/2)$ fixed above, the $O(1)$ inside the braces in the contribution of the first minimizing selection comes from the part of its $y$-Wronskian without a factor $y$ in \eqref{eq:W.y} and is uniform for $\Delta \in J$. The complete contribution of the second minimizing selection is $e^{-m^2y}O(1)$, uniformly for $\Delta \in J$, by the preceding row calculation and the local boundedness of its $\Delta$-Wronskian. The sum of all selections of larger exponential degree is also $e^{-m^2y}O(1)$, uniformly for $\Delta \in J$, by the tail bound in \eqref{eq:tail.bound}. Since $C_m > 0$ is fixed, these last two contributions can be absorbed into the $O(1)$ inside the braces in \eqref{eq:first.min.contrib}. This proves \eqref{eq:D.asymptotic} locally uniformly in $\Delta$ and concludes the proof of Lemma~\ref{lem:confluent.boundary.asymptotic}.

\subsection{Nonvanishing of \texorpdfstring{$\det H_m$}{det H\_m}}

Recall that $\xi = (\tfrac{1}{2} - \Delta)^2$, and recall the definition of $H_m(\xi)$ from \eqref{eq:H.m}. Lemma~\ref{lem:confluent.boundary.asymptotic} shows that $\det H_m(\xi)$ controls the sign of the leading large-$y$ term in $D_{2m + 1}$. The proof of Lemma~\ref{lem:H.sign.obstruction} requires a single value of $\Delta$ at which this determinant is nonzero for every $m$. The next lemma makes such a choice possible: for each fixed $m$, the zero set $\{\xi \in (0, 1/4) : \det H_m(\xi) = 0\}$ is discrete, so the countable union of these zero sets can be avoided.

\begin{lemma}[Nonvanishing of $\det H_m$]\label{lem:H.nonzero}
For every $m \in \N$, the function $\xi\mapsto \det H_m(\xi)$ is real analytic and is not identically zero on $(0,1/4)$.
\end{lemma}

\begin{proof}[Proof of Lemma~\ref{lem:H.nonzero}]
Fix $m \in \N$. The required nontriviality will follow from the behavior of $\det H_m(\xi)$ as $\xi \downarrow 0$.

The description of the digamma function in \cite[\S5.2(i)]{AskeyRoy2010} shows that its poles are $0,-1,-2,\ldots$. Hence, the nearest pole to $1/2$ is $0$, at distance $1/2$, and the Taylor series of $\psi$ about $1/2$ converges whenever the increment has absolute value less than $1/2$. Since $\xi^{1/2} < 1/2$ for $0 < \xi < 1/4$, then the function $c$ in \eqref{eq:xi.c.definition} has the convergent expansion
\begin{equation}\label{eq:c.sqrt.expansion}
c(\xi) = -2\gamma - 2\sum_{\ell=0}^{\infty} \frac{\psi^{(\ell)}(1/2)}{\ell!} (-\xi^{1/2})^{\ell}, \qquad 0 < \xi < \frac{1}{4}.
\end{equation}
By \citet[Eq.~5.15.4]{AskeyRoy2010}, $\psi^{(1)}(1/2) = \pi^2/2$, so the coefficient of $\xi^{1/2}$ in \eqref{eq:c.sqrt.expansion} is $\pi^2$. Separating the integer and noninteger powers in the convergent series \eqref{eq:c.sqrt.expansion} and differentiating that series term by term shows that, for every fixed integer $r \geq 2$,
\begin{equation}\label{eq:c.derivative.near.zero}
\begin{aligned}
\frac{c^{(r)}(\xi)}{r!}
&= \frac{\pi^2}{r!} \left(\frac{1}{2}\right)\left(\frac{1}{2} - 1\right)\cdots\left(\frac{1}{2} - r + 1\right)\xi^{\frac{1}{2} - r} + O(\xi^{\frac{3}{2} - r}) \\
&= \pi^2\binom{1/2}{r}\xi^{\frac{1}{2} - r}(1 + O(\xi))
\end{aligned}
\end{equation}
as $\xi \downarrow 0$. The next noninteger power is $\xi^{3/2}$, whose $r$th derivative is $O(\xi^{\frac{3}{2} - r})$ and is therefore $O(\xi)$ relative to the leading term. An integer power $\xi^j$ vanishes after $r$ derivatives when $j < r$; when $j \geq r$, its contribution is $O(\xi^{j - r})$, which is $O(\xi^{3/2})$ relative to $\xi^{\frac{1}{2} - r}$ because $r \geq 2$. This proves the relative error estimate in \eqref{eq:c.derivative.near.zero}. In particular,
\[
\frac{c^{(i + j)}(\xi)}{(i + j)!} = \pi^2\binom{1/2}{i + j}\xi^{\frac{1}{2} - i - j}(1 + O(\xi)), \qquad 1 \leq i,j \leq m.
\]
After $\pi^2\xi^{\frac{1}{2} - i}$ is factored from row $i$ and $\xi^{-j}$ is factored from column $j$, the $(i,j)$th entry of the remaining matrix converges to $\binom{1/2}{i + j}$. Continuity of the determinant therefore gives
\begin{equation}\label{eq:H.near.zero}
\det H_m(\xi) = (\pi^2)^m\xi^{-m^2 - m/2}\left\{\det\!\left[\binom{1/2}{i + j}\right]_{i,j=1}^m + o(1)\right\}.
\end{equation}
Indeed, the exponent of $\xi$ that has been factored out is $\sum_{i=1}^m\left(\frac{1}{2} - i\right) - \sum_{j=1}^m j = -m^2 - m/2$.

It remains to show that the determinant in braces in \eqref{eq:H.near.zero} is nonzero. For $i,j \geq 1$, the Gamma function definition, Euler's reflection formula, and the beta integral representation give
\begin{equation}\label{eq:binomial.beta.integral}
\begin{aligned}
\left|\binom{1/2}{i + j}\right|
&= \left| \frac{\Gamma(3/2)}{\Gamma(i + j + 1)\Gamma(\frac{3}{2} - i - j)} \right| \\
&= \frac{\Gamma(i + j - \frac{1}{2})\Gamma(3/2)}{\pi\,\Gamma(i + j + 1)} = \frac{1}{\pi}\int_0^1 x^{i + j - \frac{3}{2}}(1 - x)^{1/2} \, \rd x.
\end{aligned}
\end{equation}
Let $M_m$ be the matrix formed from the absolute values in \eqref{eq:binomial.beta.integral}. For every nonzero vector $\bb{v} = (v_1,\ldots,v_m) \in \R^m$,
\[
\begin{aligned}
\bb{v}^{\mathsf T}M_m \bb{v}
&= \frac{1}{\pi}\int_0^1 x^{1/2}(1 - x)^{1/2}\left(\sum_{i=1}^m v_i x^{i - 1}\right)^2 \, \rd x > 0.
\end{aligned}
\]
Thus, $M_m$ is the Gram matrix of the linearly independent functions $1,x,\ldots,x^{m - 1}$ with respect to the weight $\pi^{-1}x^{1/2}(1 - x)^{1/2}\,\rd x$, whose density is nonnegative on $[0,1]$ and strictly positive on $(0,1)$. Hence, $M_m$ is positive definite and $\det M_m > 0$. Moreover, the generalized binomial-coefficient formula in \cite[p.~154, Eq.~(5.1)]{GrahamKnuthPatashnik1994} shows that
\[
\binom{1/2}{i + j} = \frac{(\frac{1}{2})(-\frac{1}{2})\cdots(\tfrac{1}{2} - (i + j) + 1)}{(i + j)!} = (-1)^{i + j - 1} \left|\binom{1/2}{i + j}\right|, \qquad i,j \geq 1.
\]
If $D = \operatorname{diag}((-1)^1,\ldots,(-1)^m)$, then
\[
\left[\binom{1/2}{i + j}\right]_{i,j=1}^m = -D M_m D.
\]
Its determinant is therefore $(-1)^m\det M_m \neq 0$. Formula \eqref{eq:H.near.zero} shows that $\det H_m$ is not identically zero. Its real analyticity on $(0,1/4)$ follows directly from \eqref{eq:xi.c.definition}, because $\tfrac{1}{2} - \xi^{1/2} > 0$ on that interval.
\end{proof}

\subsection{A representation by a nonnegative measure}

Every entry of $H_m(\xi)$ involves a derivative $c^{(k)}(\xi)$ of order $k \geq 2$. The next lemma represents all these derivatives using one nonnegative measure, thereby converting quadratic forms associated with $H_m$ into integrals of polynomial squares. This representation will be used in the proof of Lemma~\ref{lem:H.sign.obstruction} to rule out the possibility that the matrices $-H_m$ are simultaneously positive definite for all $m\in \N$. It is obtained by extending $c$ to the upper half-plane as a Herglotz function and identifying its representing measure from its boundary values and poles.

\begin{lemma}[The representation by a nonnegative measure]\label{lem:c.derivative.measure}
Recall the definition of $c(\xi)$ from \eqref{eq:xi.c.definition}. Let $0 < \xi < 1/4$, and let $\delta_x$ denote the unit point mass at $x$. Define the nonnegative measure
\begin{equation}\label{eq:Herglotz.measure}
\rd\mu(t) \leqdef \ind_{\{t<0\}}\tanh(\pi(-t)^{1/2}) \, \rd t + \sum_{n=0}^{\infty}4(n + \tfrac{1}{2})\delta_{(n + \frac{1}{2})^2}(\rd t).
\end{equation}
Then, for every integer $k \geq 2$,
\begin{equation}\label{eq:c.derivative.measure}
\frac{c^{(k)}(\xi)}{k!} = \int_{\R}\frac{\rd\mu(t)}{(t - \xi)^{k + 1}}.
\end{equation}
\end{lemma}

\begin{proof}[Proof of Lemma~\ref{lem:c.derivative.measure}]
A function that is analytic on the upper half-plane and has nonnegative imaginary part there is called a Herglotz function (or Pick function, or Nevanlinna function). The Herglotz representation theorem, in the form recorded by \citet[Thm.~2.2(iii), Eqs.~(2.14) and~(2.15)]{GesztesyTsekanovskii2000}, states that every such function $\Phi$ can be written as
\begin{equation}\label{eq:Herglotz.statement}
\Phi(w) = \alpha + \beta w + \int_{\R}\left\{\frac{1}{t - w} - \frac{t}{1 + t^2}\right\}\,\rd\mu(t), \qquad \operatorname{Im}(w) > 0,
\end{equation}
where $\alpha \in \R$, $\beta \geq 0$, and $\mu$ is a nonnegative Borel measure satisfying $\int_{\R}(1 + t^2)^{-1}\,\rd\mu(t) < \infty$. The Stieltjes inversion formula in \cite[Thm.~2.2(iv), Eq.~(2.16)]{GesztesyTsekanovskii2000} gives
\begin{equation}\label{eq:Stieltjes.inversion}
\mu((a,b)) = \lim_{v\downarrow 0}\frac{1}{\pi}\int_a^b\operatorname{Im}\Phi(t + \ii v) \, \rd t
\end{equation}
whenever $a$ and $b$ are not atoms. In particular, \citet{GesztesyTsekanovskii2000} identify two key limits from the boundary behavior:
\begin{enumerate}[label=(\roman*), itemsep=0pt, parsep=0pt, topsep=0pt]
    \item The density of the absolutely continuous part, almost everywhere on an interval where the boundary value exists continuously, as $\pi^{-1}\operatorname{Im}\Phi(t + \ii 0)$ \cite[Thm.~2.2(v), Eq.~(2.17)]{GesztesyTsekanovskii2000};
    \item The mass at $t^{\star}$ as $\lim_{v\downarrow 0}v\operatorname{Im}\Phi(t^{\star} + \ii v)$ \cite[Thm.~2.3(iv), Eq.~(2.23)]{GesztesyTsekanovskii2000}.
\end{enumerate}

\bigskip

Define
\begin{equation}\label{eq:Phi.definition}
\Phi(w) \leqdef -2\gamma - 2\psi\Big(\tfrac{1}{2} - w^{1/2}\Big), \qquad \operatorname{Im}(w) > 0,
\end{equation}
where the principal square root is used. Note that this is the same functional form as the function $h(\Delta)$ defined in \eqref{eq:h.definition}, evaluated at $\tfrac{1}{2} - w^{1/2}$. The principal square root is analytic on the upper half-plane. If $\operatorname{Im}(w) > 0$, then $\zeta = \tfrac{1}{2} - w^{1/2}$ has strictly negative imaginary part and therefore cannot be a pole of $\psi$. Thus, $\Phi$ is analytic on the upper half-plane. The convergent expansion in \cite[Eq.~5.7.6]{AskeyRoy2010},
\[
\psi(\zeta) = -\gamma + \sum_{n=0}^{\infty}\left\{\frac{1}{n + 1} - \frac{1}{n + \zeta}\right\},
\]
shows term by term that $\operatorname{Im}\psi(\zeta) < 0$. Hence, $\operatorname{Im}\Phi(w) > 0$, so $\Phi$ is a Herglotz function.

We now identify its representing measure. If $t = -r^2 < 0$ with $r > 0$, then $(t + \ii 0)^{1/2} = \ii r$. The reflection identity for the digamma function in \cite[Eq.~5.5.4]{AskeyRoy2010} gives
\[
\psi(\tfrac{1}{2} + \ii r) - \psi(\tfrac{1}{2} - \ii r)
= -\frac{\pi}{\tan(\pi(\frac{1}{2} + \ii r))}
= \frac{\pi}{\cot(\ii\pi r)}
= -\frac{\pi}{\ii\coth(\pi r)}
= \ii\pi\tanh(\pi r),
\]
while $\psi(\tfrac{1}{2} + \ii r) = \overline{\psi(\tfrac{1}{2} - \ii r)}$ by the Schwarz reflection principle \citep[p.~211]{Conway1978}. Consequently,
\[
\operatorname{Im}\Phi(t + \ii 0) = -2\operatorname{Im}\psi(\tfrac{1}{2} - \ii r)
= -\ii\{\psi(\tfrac{1}{2} + \ii r) - \psi(\tfrac{1}{2} - \ii r)\} = \pi\tanh(\pi(-t)^{1/2}).
\]
By $\mathrm{(i)}$, the absolutely continuous part of the representing measure of $\Phi$ on $(-\infty,0)$ has density $\tanh(\pi(-t)^{1/2})$.

It remains to rule out a singular component on $(-\infty,0)$. Fix $a < b < 0$. For all sufficiently small $v_0 > 0$, the principal square root and the digamma function in \eqref{eq:Phi.definition} extend continuously to the compact set $\{t + \ii v : a \leq t \leq b, 0 \leq v \leq v_0\}$. Therefore,
\[
\operatorname{Im}\Phi(t + \ii v) \longrightarrow \pi\tanh(\pi(-t)^{1/2})
\]
uniformly for $t \in [a,b]$ as $v \downarrow 0$. In particular, these boundary values are bounded, so $\mathrm{(ii)}$ shows that no point of $(-\infty,0)$ is an atom. Applying \eqref{eq:Stieltjes.inversion} and dominated convergence gives
\[
\mu((a,b)) = \int_a^b \tanh(\pi(-t)^{1/2}) \, \rd t.
\]
Since $a$ and $b$ were arbitrary, the representing measure of $\Phi$ has no singular component on $(-\infty,0)$.

Next, let $t_n = (n + \tfrac{1}{2})^2$. As $w \to t_n$,
\[
w^{1/2} = t_n^{1/2} + \frac{w - t_n}{2t_n^{1/2}} + O((w - t_n)^2).
\]
The digamma function has a simple pole of residue $-1$ at $-n$, as recorded in \cite[\S5.2(i)]{AskeyRoy2010}. Since
\[
\frac{1}{2} - w^{1/2} = -n - \frac{w - t_n}{2t_n^{1/2}} + O((w - t_n)^2),
\]
direct substitution in \eqref{eq:Phi.definition} yields
\[
\Phi(w) = -2\gamma - 2\left( \dfrac{-1}{- \dfrac{w - t_n}{2t_n^{1/2}} + O((w - t_n)^2)} + O(1) \right) = \frac{4t_n^{1/2}}{t_n - w} + O(1).
\]
Thus, by $\mathrm{(ii)}$, the representing measure of $\Phi$ has mass $4(n + \tfrac{1}{2})$ at $t_n$.

Now, the function $\Phi$ extends analytically, with real values, through every open interval in $(0,\infty)$ that contains none of the points $t_n$. On any such interval, $\operatorname{Im}\Phi(t + \ii 0) = 0$, so \eqref{eq:Stieltjes.inversion} rules out additional support there. Expanding $\psi(\tfrac{1}{2} - w^{1/2})$ about $1/2$ and using $\psi^{(1)}(1/2) = \pi^2/2$ gives
\[
\Phi(w) = -2\gamma - 2\psi(1/2) + \pi^2 w^{1/2} + O(w)
\]
as $w \to 0$ through the upper half-plane. Consequently, $v\operatorname{Im}\Phi(\ii v) = O(v^{3/2}) \to 0$ as $v \downarrow 0$, so the representing measure has no atom at zero. The density on $(-\infty,0)$, the masses $4(n + \tfrac{1}{2})$ at the points $t_n$, the absence of support on the remaining intervals in $(0,\infty)$, and the absence of an atom at zero together identify the representing measure as
\[
\rd\mu(t) = \ind_{\{t<0\}}\tanh(\pi(-t)^{1/2}) \, \rd t + \sum_{n=0}^{\infty}4(n + \tfrac{1}{2})\delta_{(n + \tfrac{1}{2})^2}(\rd t),
\]
which is \eqref{eq:Herglotz.measure}. This measure satisfies the integrability condition $\int_{\R}(1 + t^2)^{-1}\,\rd\mu(t) < \infty$ in the Herglotz representation: the continuous density is bounded, and the contribution of the point masses for large $n$ is bounded by a constant times $\sum_{n\geq1}n^{-3}$, as explicitly shown by
\[
\int_{\R}\frac{\rd\mu(t)}{1 + t^2}
= \int_{-\infty}^0 \frac{\tanh(\pi(-t)^{1/2})}{1 + t^2} \, \rd t + \sum_{n=0}^{\infty}\frac{4(n + \tfrac{1}{2})}{1 + (n + \tfrac{1}{2})^4}
\leq \int_{-\infty}^0 \frac{\rd t}{1 + t^2} + \sum_{n=0}^{\infty} \frac{4}{(n + \frac{1}{2})^3} < \infty.
\]

The integral on the right-hand side of \eqref{eq:Herglotz.statement} is analytic away from the support of $\mu$. In particular, both sides extend analytically through $(0,t_0) \equiv (0,1/4)$, where $\Phi(\xi) = c(\xi)$. For every $k \geq 2$, the affine term $\alpha + \beta\xi$ and the $\xi$-independent subtraction term $-t/(1 + t^2)$ in the representation \eqref{eq:Herglotz.statement} make no contribution, while $(\rd^k / \rd \xi^k) (t - \xi)^{-1} = k! (t - \xi)^{-k-1}$. Therefore, for every $k \geq 2$ and $0 < \xi < 1/4$,
\[
\frac{c^{(k)}(\xi)}{k!} = \int_{\R}\frac{\rd\mu(t)}{(t - \xi)^{k + 1}}.
\]
Differentiation under the integral is justified absolutely. For fixed $\xi \in (0,1/4)$, the support of $\mu$ stays a strictly positive distance from $\xi$. On the continuous part $(-\infty,0)$, the integrand is $O(|t|^{-k - 1})$ as $t \to -\infty$, and the sum of the absolute values over the point masses for large $n$ is $O(\sum_{n\geq1}n^{-2k - 1})$. This proves \eqref{eq:c.derivative.measure} and concludes the proof of the lemma.
\end{proof}

\subsection{Proof of Lemma~\ref{lem:H.sign.obstruction}}

For each $m \in \N$, Lemma~\ref{lem:H.nonzero} implies that the zeros of $\det H_m$ are discrete. The map $\Delta \mapsto (\tfrac{1}{2} - \Delta)^2$ is a strictly decreasing homeomorphism from $I$ onto its image, so the inverse images of these zero sets are also discrete, and hence countable, in $I$. The union of these inverse-image zero sets over all $m \in \N$ is countable, whereas $I$ is uncountable. Consequently, there is a $\Delta_0 \in I$ such that, with $\xi_0 \leqdef (\tfrac{1}{2} - \Delta_0)^2$,
\begin{equation}\label{eq:generic.point}
\det H_m(\xi_0) \neq 0, \qquad m \in \N.
\end{equation}

At the point $\xi_0$ chosen in \eqref{eq:generic.point}, every number $(-1)^m\det H_m(\xi_0)$ is nonzero. Suppose, for contradiction, that all these numbers are nonnegative. Then
\begin{equation}\label{eq:H.leading.minors}
(-1)^m\det H_m(\xi_0) > 0, \qquad m \in \N.
\end{equation}
For each fixed $M \in \N$, the $m$th leading principal minor of $-H_M(\xi_0)$ is $(-1)^m\det H_m(\xi_0)$ for $1 \leq m \leq M$. Therefore, Sylvester's criterion and \eqref{eq:H.leading.minors} imply that $-H_M(\xi_0)$ is positive definite for every $M$.

With $u = (t - \xi_0)^{-1}$, Lemma~\ref{lem:c.derivative.measure} gives, for every $\bb{v} = (v_1,\ldots,v_m) \in \R^m$,
\[
\bb{v}^{\mathsf T}H_m(\xi_0)\bb{v}
= \sum_{i,j=1}^m v_i v_j \frac{c^{(i + j)}(\xi_0)}{(i + j)!}
= \sum_{i,j=1}^m v_i v_j\int_{\R}u^{i + j + 1} \, \rd\mu(t)
= \int_{\R}u\left(\sum_{i=1}^m v_i u^i\right)^2 \, \rd\mu(t).
\]
Since $-H_m(\xi_0)$ is positive definite, the last integral is strictly negative whenever $\bb{v} \neq \bb{0}$. Let $R$ be any nonzero polynomial of degree $d$, take $m = d + 1$, and choose $v_1,\ldots,v_m$ to be the coefficients of the polynomial $uR(u) = \sum_{i=1}^m v_i u^i$. Then $\bb{v} \neq \bb{0}$, and the last integral becomes $\int_{\R}u^3R(u)^2\,\rd\mu(t)$. Consequently,
\begin{equation}\label{eq:signed.measure.negative}
\int R(u)^2\,\rd\sigma(u) < 0,
\end{equation}
where $\sigma$ is the push-forward under $t \mapsto (t - \xi_0)^{-1}$ of the signed measure $(t - \xi_0)^{-3}\,\rd\mu(t)$; explicitly,
\[
\int_{\R} f(u)\,\rd\sigma(u) \leqdef \int_{\R}\frac{1}{(t - \xi_0)^3}f\left(\frac{1}{t - \xi_0}\right)\,\rd\mu(t).
\]
Note that the signed measure $\sigma$ has finite total variation:
\[
\begin{aligned}
\int_{\R} \rd|\sigma|(u)
&= \int_{-\infty}^0 \frac{\tanh(\pi(-t)^{1/2})}{|t - \xi_0|^3} \, \rd t + \sum_{n=0}^{\infty} \frac{4(n + \tfrac{1}{2})}{|(n + \tfrac{1}{2})^2 - \xi_0|^3} \\
&\leq \int_{-\infty}^0 \frac{1}{|t - \xi_0|^3} \, \rd t + \frac{2}{(\frac{1}{4} - \xi_0)^3} + \sum_{n=1}^{\infty}\frac{4(n + \frac{1}{2})}{n^3(n + 1)^3} < \infty.
\end{aligned}
\]
Here, the last inequality uses the fact that $(n + \tfrac{1}{2})^2 - \xi_0 > n(n + 1)$ for all $n \geq 1$.

The support of $\sigma$ is contained in the compact set
\[
\left[-\frac{1}{\xi_0},0\right]\cup\left\{\frac{1}{(n + \frac{1}{2})^2 - \xi_0}:n \geq 0\right\}.
\]
The largest strictly positive point in the support of $\sigma$, $u_0 \leqdef \frac{1}{\frac{1}{4} - \xi_0}$, is isolated, and \eqref{eq:Herglotz.measure} gives
\[
\sigma(\{u_0\}) = 2u_0^3 > 0.
\]
Since $u_0$ is isolated from the rest of the support of $\sigma$, choose a continuous function on a compact interval containing the support that equals one at $u_0$ and vanishes on the rest of the support. By the Weierstrass approximation theorem, there are polynomials $R_j$ converging uniformly to this function. In particular, $R_j$ is nonzero for all sufficiently large $j$, because $R_j(u_0) \to 1$. Since $\sigma$ has finite total variation, the uniform convergence justifies interchanging the limit and the integral, yielding
\[
\lim_{j\to\infty}\int R_j(u)^2\,\rd\sigma(u) = \sigma(\{u_0\}) > 0,
\]
which contradicts \eqref{eq:signed.measure.negative}. Thus, \eqref{eq:H.leading.minors} cannot hold for every $m$. In view of \eqref{eq:generic.point}, there is an $m \in \N$ such that $(-1)^m\det H_m(\xi_0) < 0$, as required. This concludes the proof of Lemma~\ref{lem:H.sign.obstruction}.

\end{appendices}

\section*{Data availability statement}
\addcontentsline{toc}{section}{Data availability statement}

No data were used or generated in this study.

\section*{Statement of AI use}
\addcontentsline{toc}{section}{Statement of AI use}

ChatGPT 5.6 Sol assisted in the initial discovery of the mathematical arguments. The authors completely reworked and independently validated all proofs, and assume full responsibility for their accuracy and rigor.

\section*{Funding}
\addcontentsline{toc}{section}{Funding}

F.\ Ouimet is supported by the Natural Sciences and Engineering Research Council of Canada (Discovery Grant RGPIN-2026-04471, Discovery Launch Supplement DGECR-2026-00449).

\addcontentsline{toc}{section}{References}

\setlength{\bibsep}{0pt plus 0ex}


\end{document}